\documentclass{amsart}

\usepackage{amsmath,amssymb,amsthm}
\usepackage{bbm,mathrsfs}
\usepackage[all]{xy}
\makeatletter

\@addtoreset{equation}{section}
\makeatother

\newcommand{\mathz}{\mathbb{Z}}
\newcommand{\mathr}{\mathbb{R}}

\newcommand{\G}{\mathbbmss{G}}
\newcommand{\U}{\mathbbmss{U}}
\newcommand{\B}{\mathbbmss{B}}
\newcommand{\T}{\mathbbmss{T}}
\renewcommand{\P}{\mathbbmss{P}}
\newcommand{\N}{\mathbbmss{N}}
\renewcommand{\O}{\mathcal{O}}
\newcommand{\F}{\mathbbmss{F}}

\newcommand{\ev}{{\mathsf{ev}}}
\newcommand{\Gev}{{\G_{\ev}}}
\newcommand{\Gevv}{\G_{\ev}}

\newcommand{\Bev}{{\B_{\ev}}}
\newcommand{\Tev}{{\T_{\ev}}}
\newcommand{\Uevp}{{\U^+_{\ev}}}
\newcommand{\Bevp}{{\B^+_{\ev}}}

\newcommand{\g}{\mathfrak{g}}
\newcommand{\h}{\mathfrak{h}}
\newcommand{\m}{\mathfrak{m}}
\newcommand{\uu}{\mathfrak{u}}
\newcommand{\q}{\mathfrak{q}}

\newcommand{\GL}{\mathbbmss{GL}}
\newcommand{\GLL}{\mathrm{GL}}
\newcommand{\SL}{\mathbbmss{SL}}
\newcommand{\SLL}{\mathrm{SL}}
\newcommand{\SpO}{\mathbbmss{SpO}}
\newcommand{\Q}{\mathbbmss{Q}}
\newcommand{\Ga}{G_{\mathsf{a}}}
\newcommand{\Gm}{G_{\mathsf{m}}}
\newcommand{\Gao}{G_{\mathsf{a}}^-}

\newcommand{\LL}{\mathtt{L}}
\newcommand{\RR}{\mathtt{R}}

\newcommand{\X}{\mathsf{X}}
\newcommand{\R}{\mathsf{R}}
\newcommand{\Z}{\mathcal{Z}}

\newcommand{\gpl}{\mathsf{g.l.}}

\newcommand{\hy}{\mathsf{hy}}
\newcommand{\Dist}{\mathsf{Dist}}
\newcommand{\Lie}{\mathsf{Lie}}
\newcommand{\Ker}{\mathsf{Ker}}

\newcommand{\Hom}{{\mathsf{Hom}}}
\newcommand{\Aut}{{\mathsf{Aut}}}
\newcommand{\End}{{\mathsf{End}}}

\newcommand{\Mat}{\mathsf{Mat}}
\newcommand{\Simple}{\mathsf{Simple}}

\newcommand{\Fr}{\mathsf{Fr}}
\newcommand{\id}{\mathsf{id}}
\newcommand{\coad}{\mathsf{coad}}
\newcommand{\ad}{\mathsf{ad}}
\newcommand{\Ad}{\mathsf{Ad}}

\newcommand{\Deltaa}{\triangle}
\newcommand{\eve}{{\bar{0}}}
\newcommand{\odd}{{\bar{1}}}

\newcommand{\cotens}{\mathbin{\Box}}
\newcommand{\hits}{\rightharpoonup}
\newcommand{\tint}{\begingroup\textstyle\int\endgroup}
\newcommand{\ttint}{\mathcal{I}}
\newcommand{\triact}{\mathbin{\triangleright}}

\newcommand{\acthy}{\hits}
\newcommand{\acthydual}{\cdot}

\newcommand{\ind}{\mathsf{ind}}
\newcommand{\coind}{\mathsf{coind}}
\newcommand{\rad}{\mathsf{rad}}
\newcommand{\soc}{\mathsf{soc}}

\newcommand{\s}{\mathcal{S}}
\newcommand{\Cl}{\mathsf{Cl}}
\newcommand{\topp}{\mathsf{top}}

\newcommand{\e}{\mathsf{e}}
\newcommand{\bmu}{{\boldsymbol\mu}}

\newcommand{\bigset}[2]{\left\{#1\vphantom{#2}\;\right|\;\left.#2\vphantom{#1}\right\}}

\renewcommand{\det}{\mathsf{det}}
\newcommand{\Ber}{\mathsf{Ber}}
\renewcommand{\dim}{\mathsf{dim}}
\newcommand{\dist}{\alpha}
\newcommand{\adm}{special }

\newcommand{\qedd}{\hfill$\blacksquare$}

\newcommand{\defnote}[1]{{\it #1}}

\newtheorem{thm}{Theorem}[section]
\newtheorem{lemm}[thm]{Lemma}
\newtheorem{prop}[thm]{Proposition}
\newtheorem{cor}[thm]{Corollary}
\newtheorem{asum}[thm]{Assumption}
\theoremstyle{definition}
\newtheorem{deff}[thm]{Definition}
\newtheorem{rem}[thm]{Remark}
\newtheorem{ex}[thm]{Example}

\begin{document}
\title[Frobenius kernels of algebraic supergroups]{Frobenius kernels of algebraic supergroups and Steinberg's tensor product theorem}
\author[T.~Shibata]{Taiki Shibata}
\address{Department of Applied Mathematics,
Okayama University of Science,
1-1 Ridai-cho Kita-ku Okayama-shi, Okayama 700-0005, Japan}
\email{shibata@ous.ac.jp}
\dedicatory{Dedicated to Professor~Akira~Masuoka on the occasion of his~60th~birthday.}

\date{\today}


\begin{abstract}
For a split quasireductive supergroup $\G$ defined over a field,
we study structure and representation of Frobenius kernels $\G_r$ of $\G$ and
we give a necessary and sufficient condition for $\G_r$ to be unimodular in terms of the root system of $\G$.
We also establish Steinberg's tensor product theorem for $\G$ under some natural assumptions.
\end{abstract}

\maketitle

\setcounter{tocdepth}{2}
\tableofcontents

\section{Introduction}\label{sec:Intro}
Structure and representation of algebraic group schemes (especially, connected and split reductive groups)
over a field have been well studied (see \cite{Jan03,Mil17} for example)
and provide applications in many areas such as combinatorics or number theory.
Over an algebraically closed field $\Bbbk$ of characteristic zero,
the Lie algebra $\Lie(G)$ of a connected and split reductive group (scheme) $G$ strongly reflects many properties of $G$ (see \cite{Hoc70})
and becomes a fundamental tool for studying representations of $G$.
For example, it is known that there exists a category equivalence between
the category of left $G$-modules and the category of locally finite left $\Lie(G)$-$T$-modules,
where $T$ denotes a split maximal torus of $G$.
Here, we say that a $\Lie(G)$-module $M$ is \defnote{$\Lie(G)$-$T$-module}
if the restricted $\Lie(T)$-module structure on $M$ arises from some $T$-module structure on it.
In particular, we can show that for a dominant weight $\lambda$,
the simple left $G$-module $L(\lambda)$ of highest weight $\lambda$
coincides with the induced representation $\ind^G_B(\Bbbk^\lambda)$ of the one-dimensional $T$-module $\Bbbk^\lambda$ of weight $\lambda$,
where $B$ denotes a fixed Borel subgroup of $G$.
The character of $L(\lambda)$ is explicitly given by Weyl's character formula.

On the other hand, over a field $\Bbbk$ of positive characteristic,
the situation is more complicated, since the simple left $G$-module $L(\lambda)$ may be a proper submodule of $\ind^G_B(\Bbbk^\lambda)$ in general.
In \cite{Tak74}, Takeuchi studied the \defnote{hyperalgebra} $\hy(G)$ of $G$
which is a natural refinement of the universal enveloping algebra $\mathcal{U}(\Lie(G))$ of $\Lie(G)$.
Note that, $\hy(G)$ is isomorphic to $\mathcal{U}(\Lie(G))$ as (cocommutative) Hopf algebras if $\mathsf{char}(\Bbbk)=0$.
By Hopf-algebraic method, as in the Lie algebra case, he showed $\hy(G)$ strongly reflects many properties of $G$ (see \cite{Tak74,Tak75,Tak83}).
There also holds a category equivalence between
the category of left $G$-modules and the category of locally finite left $\hy(G)$-$T$-modules (see \cite[Part~II, Chapter~1]{Jan03} for example).
Over a perfect field $\Bbbk$ of positive characteristic $p$,
for each positive integer $r$,
the kernel $G_r$ of the $r$-th iterated Frobenius morphism $\Fr^r:G\to G$,
called the \defnote{$r$-th Frobenius kernel} of $G$,
is a fundamental and powerful tool for studying $G$.
By definition, we have an ascending chain $G_1\subset G_2\subset\cdots \subset G$ of normal subgroup of $G$ and $\hy(G)=\varinjlim_r \hy(G_r)$.
Moreover, it is known that all Frobenius kernels $G_r$ are unimodular,
that is,
there exists non-zero two-sided integral for $G_r$, see Definition~\ref{def:unimo}.
Using the categorical equivalence of modules mentioned above,
we can show \defnote{Steinberg's tensor product theorem} (\cite[Part~II, Corollary~3.17]{Jan03})
which states that, if a dominant weight $\lambda$ admits an expression
$\lambda=\lambda_0+p\lambda_1+\cdots+p^m\lambda_m$
with all $\lambda_i$ dominant and $p$-restricted for $G$, then
\[
L(\lambda) \cong L(\lambda_0)\otimes L(\lambda_1)^{[1]}\otimes L(\lambda_2)^{[2]}\otimes\cdots\otimes L(\lambda_m)^{[m]}
\]
as left $G$-modules.
Here, $p$-restriction means $0\leq\langle\lambda_i,\alpha^\vee\rangle<p$ for every simple root $\alpha$ of $G$.
In particular, the character of $L(\lambda)$ can be calculated by the product of the character of $L(\lambda_r)^{[r]}$.
Note that, if we write the character of a $G$-module $M$ as $\sum_{\lambda} \dim(M^\lambda) \e^\lambda$,
then the character of $\Fr^r$-twisted $G$-module $M^{[r]}$ is given by $\sum_{\lambda} \dim(M^\lambda) \e^{p^r\lambda}$.
For weights admitting such an expansion, the decomposition reduces the study of $L(\lambda)$ to simple left $G$-modules with $p$-restricted weights.
\medskip

In recent years, supergeometries and superalgebras have attracted much attention.
The word ``super'' is a synonym of ``graded by the group $\mathz_2$ of order two'' (see Section~\ref{sec:hopf}).
The symmetric tensor category of vector spaces is generalized by the category of superspaces (i.e., $\mathz_2$-graded vector spaces) with
the familiar tensor product and \defnote{supersymmetry}.
The classification of finite dimensional simple Lie superalgebras over an algebraically closed field of characteristic zero was done by Kac \cite{Kac77}.
Since then, many authors have studied the corresponding algebraic supergroup
(\cite{Kos77,Kos83,Pen88-1,Bos90,BruKuj03,Mas05,Zub06,CarCasFio11,Vis11,FioGav12} for example).
Here, an \defnote{algebraic supergroup} $\G$ is a representable functor from the category of commutative superalgebras to the category of groups;
the representing object $\O(\G)$ is a finitely generated commutative Hopf superalgebra.
In this paper, as the super-analogue of the connected and split reductive groups,
we study an algebraic supergroup $\G$ whose ``even part'' $\Gev$ is a connected and split reductive group,
called a \defnote{split quasireductive supergroup} (\cite{Shi20,Shi21}, see also \cite{Ser11,GriZub16}),
over a field.
The class of split quasireductive supergroup has many important algebraic supergroups,
for example,
the general linear supergroups $\GL(m|n)$,
the queer supergroups $\Q(n)$,
the periplectic supergroups $\P(n)$,
Chevalley supergroups of classical type
(including special linear supergroups $\SL(m|n)$ and ortho-symplectic supergroups $\SpO(m|n)$) due to Fioresi and Gavarini \cite{FioGav12}, etc.
As in the non super-situation,
if the base field is of characteristic zero,
then representation theory of a split quasireductive supergroup $\G$ is essentially the same as the Lie superalgebra $\Lie(\G)$ of $\G$.

In this paper, we are interested in modular representation theory of split quasireductive supergroup,
that is, the case when the characteristic of the base field $\Bbbk$ is positive.
As in the non super-situation,
we can define Frobenius kernels $\G_r$ of a split quasireductive supergroup $\G$
and these are also powerful tool for studying $\G$.
For example, using Frobenius kernels of $\GL(m|n)$,
Marko and Zubkov \cite{MarZub18} provided the linkage principle and described blocks of $\GL(m|n)$.
In this paper, we give a necessary and sufficient condition for $\G_r$ to be unimodular in terms of the root system of $\G$.
Thus, in contrast to the non super-situation, there exists a non unimodular $\G_r$ (see Example~\ref{ex:non-unimo_p}).

Recently, it is shown by several authors that
Steinberg's tensor product theorem holds for $\GL(m|n)$ \cite{Kuj03}, $\Q(n)$ \cite{BruKle03} and $\SpO(m|n)$ \cite{ShuWan08}.
For the classification of simple supermodules for simply connected Chevalley supergroups of type $D(2|1;\zeta)$, $G(3)$ and $F(3|1)$, see \cite{CheShuWan19}.
For a split quasireductive supergroup $\G$ in general,
it has been shown in \cite{MasShi17} that
there exists a category equivalence between
the category of left $\G$-supermodules and the category of locally finite left $\hy(\G)$-$T$-supermodules.
Moreover, all simple left $\G$-supermodules have been systematically constructed in \cite{Shi20}.
Therefore, it is natural to ask whether Steinberg's tensor product theorem holds for $\G$, in general.
To answer this question, one encounters the following two difficulties:
\begin{enumerate}
\item\label{diff(1)}
not all simple left $\G$-supermodules are \defnote{absolutely simple},
that is,
there exists a simple left $\G$-supermodule which is no longer simple after base change to some field extension of $\Bbbk$
(see Definition~\ref{def:abs-irr});
\item\label{diff(2)}
the root system of $\G$ is ill-behaved (see Example~\ref{ex:root-sys}\eqref{ex:root-sys(4)} for example)
without suitable extra conditions on the ``odd part'' of $\G$.
\end{enumerate}
Note that, in the non super-situation, \eqref{diff(1)} never happens.

In this paper,
we prove that these difficulties \eqref{diff(1)} and \eqref{diff(2)} can be overcome by attaching appropriate natural conditions.
We first show that if the root system $\Deltaa$ of $\G$ does not contain
the unit $\mathbf{0}$ of the character group $\X(T)$ of a fixed split maximal torus of $\Gev$,
then all simple left $\G$-supermodules are absolutely simple (Proposition~\ref{prp:T=Tev-abs-simple}).
Therefore, to resolve \eqref{diff(1)},
we assume that \eqref{diff(1)}$'$ the base field $\Bbbk$ is algebraically closed if $\mathbf{0}\in\Deltaa$.
To resolve \eqref{diff(2)},
we also assume that \eqref{diff(2)}$'$ the root system $\Deltaa$ of $\G$ has a \defnote{\adm base}
(i.e., an existence of even/odd ``simple roots'' of $\Deltaa$), see Definition~\ref{asu:gen} for the detail.
The root systems of $\GL(m|n)$, $\Q(n)$ and $\P(n)$ have \adm bases for the choices in Example~\ref{ex:bases}.
Under these natural assumptions \eqref{diff(1)}$'$ and \eqref{diff(2)}$'$,
we establish Steinberg's tensor product theorem for $\G$ (Corollary~\ref{cor:main});
the iterated formula requires the subsequent weights to be $p$-restricted for $\Gev$.

\subsection*{Organization of this paper}
This paper is organized as follows:
In Section~\ref{sec:Prelim},
we review some basic definitions and results for Hopf superalgebras and algebraic supergroups defined over a field.
The Lie superalgebra $\Lie(\G)$ and the super-hyperalgebra $\hy(\G)$ of an algebraic supergroup $\G$ are reviewed in Section~\ref{sec:Lie(G)-hy(G)}.

In Section~\ref{sec:Quasired},
we define the notion of split quasireductive supergroups which is the main object of study in this paper.
Since the even $\Gev$ part of a split quasireductive supergroup $\G$ is a connected and split reductive group (scheme) by definition,
we fix a split maximal torus $T$ of $\Gev$.
Thus, inside of the character group $\X(T)$ of $T$,
we can define the root system $\Deltaa$ of $\G$ with respect to $T$ (Section~\ref{sec:root-sys})
whose non-zero roots have a parity $\Deltaa\setminus\{\mathbf{0}\}=\Deltaa_\eve\cup\Deltaa_\odd$.
Over a perfect field, we study structures of Frobenius kernels $\G_r$ of $\G$ in Section~\ref{sec:Frob-ker}.
In particular, we describe a basis of the Hopf superalgebra $\O(\G_r)$ (see \eqref{eq:basis_of_O(G_r)})
and establish the PBW theorem for the super-hyperalgebra $\hy(\G_r)$ of $\G_r$ (Theorem~\ref{prp:PBW-Gr}).

In Section~\ref{sec:Unimo},
we discuss the unimodularity of Frobenius kernels of a split quasireductive supergroup over a perfect field.
First, we review basic definitions and results for left/right (co)integrals of Hopf superalgebras (Section~\ref{sec:(co)int}).
A left (resp.~right) integral for an algebraic supergroup $\G$ is defined to be
a left (resp.~right) cointegral on the corresponding Hopf superalgebra $\O(\G)$.
We say that $\G$ is unimodular if there exists a two-sided (i.e., left and right) integral for $\G$.
In \cite{MasShiShi22}, it has been shown that $\G$ has a left (resp.~right) integral if and only if its even part $\Gev$ does.
Thus, by Sullivan's result \cite{Sul72}, if the characteristic of the base field is zero,
then a connected algebraic supergroup $\G$ has a left (or right) integral if and only if $\Gev$ is reductive.
Over a field of characteristic zero,
we give a necessary and sufficient condition for a split quasireductive supergroup $\G$ to be unimodular
in terms of its root system $\Deltaa$ (Theorem~\ref{prp:unimo-char=0} and Corollary~\ref{cor:quasi-unimo}).
It is known that being unimodular is equivalent to that the distinguished group-like element is trivial (cf \cite[Chapter~10]{Rad12}).
In Section~\ref{sec:fin-normal},
we investigate properties of the distinguished group-like element of a finite normal super-subgroup of an algebraic supergroup, in general.
In Section~\ref{sec:unimod-Frob-ker},
we study unimodularity of Frobenius kernels $\G_r$ of a split quasireductive supergroup $\G$ defined over a perfect field.
Note that, $\G_r$ always has an integral, since $\G_r$ is finite (i.e., $\O(\G_r)$ is finite-dimensional).
Using the result \cite[Proposition~6.11]{ZubMar16} by Zubkov and Marko,
we get an explicit description of the distinguished group-like element of $\G_r$,
and hence we give a necessary and sufficient condition for all $\G_r$ to be unimodular
in terms of $\Deltaa$ (Theorem~\ref{prp:Frob-unimo} and Corollary~\ref{cor:Frob-unimo}).

In Section~\ref{sec:Ste},
we establish Steinberg's tensor product theorem for a split quasireductive supergroup $\G$ under some natural assumptions.
In Section~\ref{sec:simple-G-supermods},
we review construction of simple $\G$-supermodules $L(\lambda)$ ($\lambda\in\X(T)^\flat$) given in \cite{Shi20}.
In the super-situation, not all of simple $\G$-supermodules are absolutely simple (see Example~\ref{ex:non-abs-irr}).
We show that if $\Deltaa$ does not contain the unit $\mathbf{0}$ of $\X(T)$,
then all simple $\G$-supermodules are absolutely simple (Proposition~\ref{prp:T=Tev-abs-simple}).
In Section~\ref{sec:simple-G_r-supermods},
we construct simple $\G_r$-supermodules $L_r(\lambda)$ ($\lambda\in\X(T)$) and show that
$L_r(\lambda)$ coincides with the $\G_r$-top of the ``highest weight module'' $M_r(\lambda)$ of weight $\lambda$ (Proposition~\ref{prp:L_r=M_r/Rad}).
In Section~\ref{sec:odd-bases},
since the root system $\Deltaa$ of $\G$ is somewhat ill-behaved,
we introduce the notion of a \defnote{\adm base} of $\Deltaa$ (see Definition~\ref{asu:gen}).
We see that typical examples of split quasireductive supergroups have bases of its root systems (Example~\ref{ex:bases}).
The rest of Section~\ref{sec:odd-bases}, we assume that $\Deltaa$ has a \adm base.
The set of all $p^r$-restricted weights for $\G$ is denoted by $\X_r(T)^\flat$,
where $p$ is the characteristic of the base field (see Definition~\ref{def:rest-weight}).
Then we show that the simple $\G$-supermodule $L(\lambda)$ of highest weight $\lambda\in \X_r(T)^\flat$
coincides with $\hy(\G_r)\acthy L(\lambda)^\lambda$,
where $L(\lambda)^\lambda$ is the $\lambda$-weight space of $L(\lambda)$ (Lemma~\ref{prp:key-lemma}).
Because of the existence of a non absolutely simple $\G$-supermodule,
in Section~\ref{sec:simple-G-G_r-supermods},
we assume that the base field is algebraically closed if $\mathbf{0}\in\Deltaa$.
This assumption ensures that the even endomorphism rings used in the evaluation-map argument for Theorem~\ref{thm:main} are $\Bbbk$ (see Remark~\ref{rem:need}).
Using Proposition~\ref{prp:key-prp},
as in the non super-situation, we can establish Steinberg's tensor product theorem for $\G$ (Theorem~\ref{thm:main} and Corollary~\ref{cor:main}).

This version corrects errors in the published article (Algebras and Representation Theory \textbf{27} (2024), 927--959).

\subsection*{Acknowledgements}
The author thanks the anonymous referees for their helpful comments that improved the quality of the manuscript.
The author is supported by JSPS KAKENHI Grant Numbers JP19K14517 and JP22K13905.

\section{Preliminaries}\label{sec:Prelim}
Throughout this paper, $\Bbbk$ denotes a fixed base field of characteristic different from $2$.
The unadorned $\otimes$ is the tensor product over $\Bbbk$.
In this section, we fix notations and collect some known results for Hopf superalgebras and supergroups.

\subsection{Hopf superalgebras}\label{sec:hopf}
Let $\mathz_2=\{\eve,\odd\}$ be the additive group of order two.
The group algebra $\Bbbk\mathz_2$ of $\mathz_2$ over $\Bbbk$ has a unique Hopf algebra structure
and a right $\Bbbk\mathz_2$-comodule is naturally regarded as $\mathz_2$-graded vector space.
The category $\mathcal{C}$ of right $\Bbbk\mathz_2$-comodules forms a monoidal category by the tensor product $\otimes$ over $\Bbbk$.
Namely, the unit object is $\Bbbk=\Bbbk\oplus0$ and
$(V\otimes W)_\epsilon = \bigoplus_{a,b\in\mathz_2,a+b=\epsilon} V_a\otimes W_b$ ($\epsilon\in\mathz_2$)
for right $\Bbbk\mathz_2$-comodules $V$ and $W$.
For a homogeneous element $0\neq v\in V_\eve \cup V_\odd$ of $V\in\mathcal{C}$,
we let $|v|$ denote the degree of $v$, called the \defnote{parity} of $v$.
We say that $V$ is \defnote{purely even} if $V=V_\eve$.
For simplicity,
when we use the symbol $|v|$, we always assume that $v$ is homogeneous.
The following \defnote{supersymmetry} ensures that the category $\mathcal{C}$ is symmetric:
\[
c_{V,W}:V\otimes W\longrightarrow W\otimes V;
\quad
v\otimes w \longmapsto (-1)^{|v||w|}w\otimes v.
\]
An object of $\mathcal{C}$ is called a \defnote{superspace}.

For a superspace $V$, we define $\Pi V\in\mathcal{C}$
by letting $(\Pi V)_\epsilon=V_{\epsilon+\odd}$ for $\epsilon\in\mathz_2$.
For simplicity, we put $\Pi^\eve V:=V$ and $\Pi^\odd V:=\Pi V$.
We denote by $\Hom_\Bbbk(V,W)$ the set of all parity preserving morphism from $V$ to $W$ in $\mathcal{C}$.
We define a superspace $\underline{\Hom}_\Bbbk(V,W)$ by $\underline{\Hom}_\Bbbk(V,W)_\epsilon:=\Hom_\Bbbk(\Pi^\epsilon V,W)$ for $\epsilon\in\mathz_2$.
As usual, we set $\End_\Bbbk(V):=\Hom_\Bbbk(V,V)$ and $\underline{\End}_\Bbbk(V):=\underline{\Hom}_\Bbbk(V,V)$.
For a superspace $V$,
we put $V^*:=\underline{\Hom}_\Bbbk(V,\Bbbk)$,
called the \defnote{dual superspace} of $V$.

A \defnote{superalgebra} (resp.~\defnote{supercoalgebra/Hopf superalgebra/Lie superalgebra}) is
an algebra (resp.~coalgebra/Hopf algebra/Lie algebra) object in the symmetric monoidal category $\mathcal{C}$.
Our convention for Lie superalgebras includes the identity $[v,[v,v]]=0$ for every odd element $v$.
\begin{ex} \label{ex:Mat}
For positive integers $m$ and $n$,
the set of all matrices of size $m\times n$ whose entries are in $\Bbbk$ is denoted by $\Mat_{m,n}(\Bbbk)$.
Then we can regard $\Mat_{m|n}(\Bbbk):=\Mat_{m+n,m+n}(\Bbbk)$ as a superspace by letting
\begin{eqnarray*}
\Mat_{m|n}(\Bbbk)_\eve
&:=&
\bigset{
\left(\begin{array}{c|c}x_{00} & O \\ \hline O & x_{11}\end{array}\right)}{
x_{00}\in \Mat_{m,m}(\Bbbk),x_{11}\in\Mat_{n,n}(\Bbbk) }, \\
\Mat_{m|n}(\Bbbk)_\odd
&:=&
\bigset{
\left(\begin{array}{c|c}O & x_{01} \\ \hline x_{10} & O\end{array}\right)}{
x_{01}\in \Mat_{m,n}(\Bbbk),x_{10}\in\Mat_{n,m}(\Bbbk) }.
\end{eqnarray*}
The usual matrix multiplication makes $\Mat_{m|n}(\Bbbk)$ into a superalgebra.
For a finite dimensional superspace $V$,
we can identify $\End_\Bbbk(V)$ (resp.~$\underline{\End}_\Bbbk(V)$) with $\Mat_{m|n}(\Bbbk)_\eve$ (resp.~$\Mat_{m|n}(\Bbbk)$),
where $m=\dim(V_\eve)$ and $n=\dim(V_\odd)$.
\qedd
\end{ex}

For a supercoalgebra $C=(C,\Delta_C,\varepsilon_C)$,
the Heyneman-Sweedler notation, such as
$\Delta_C(c)=\sum_c c_{(1)}\otimes c_{(2)}$ and
\[
\sum_c \Delta_C(c_{(1)})\otimes c_{(2)} = \sum_c c_{(1)}\otimes c_{(2)}\otimes c_{(3)} = \sum_c c_{(1)} \otimes \Delta_C(c_{(2)})
\]
is used to express the comultiplication $\Delta_C:C\to C\otimes C$ of $c\in C$.
Note that, $\varepsilon_C(c)=0$ for all $c\in C_\odd$.
For a right $C$-supercomodule $V$ with the structure map $\rho_V:V\to V\otimes C$,
we also use the Heyneman-Sweedler notation to express the coaction, such as $\rho_V(v) = \sum_v v_{(0)}\otimes v_{(1)}$ for $v\in V$.
For right $C$-supercomodules $V$ and $W$, the set of all parity preserving right $C$-supercomodule map form $V$ to $W$ is denoted by $\Hom^C(V,W)$,
and define a superspace $\underline{\Hom}^C(V,W)$ so that $\underline{\Hom}^C(V,W)_\epsilon:=\Hom^C(\Pi^\epsilon V,W)$ for $\epsilon \in\mathz_2$.

Let $H$ be a Hopf superalgebra.
By definition, we get
\[
\Delta_H(ab) = \sum_{a,b} (-1)^{|a_{(2)}||b_{(1)}|} a_{(1)}b_{(1)} \otimes a_{(2)}b_{(2)}
\quad\text{for}\quad
a,b\in H,
\]
where $\Delta_H$ is the comultiplication of $H$.
In this paper, the antipode of $H$ is denoted by $\s_H:H\to H$.
\begin{ex} \label{ex:exterior}
For a vector space $V$, the exterior algebra $H=\bigwedge(V)$ of $V$ over $\Bbbk$ naturally becomes a commutative superalgebra.
Moreover, $H$ forms a Hopf superalgebra by letting:
$\Delta_H(v)=v\otimes1+1\otimes v$, $\varepsilon_H(v)=0$ and $\s_H(v)=-v$ for $v\in V$.
Note that, $H$ is cocommutative.
\qedd
\end{ex}
As in the non super-situation,
if a Hopf superalgebra $H$ is commutative or cocommutative,
then the antipode $\s_H:H\to H$ of $H$ satisfies $\s_H^2=\id_H$.
In particular, $\s_H$ is bijective.

\begin{deff} \label{def:group-like}
Let $H$ be a Hopf superalgebra.
A non-zero element $g$ of $H$ is called a \defnote{group-like elements} of $H$ if it satisfies $g\in H_\eve$ and $\Delta_H(g)=g\otimes g$.
The set of all group-like elements of $H$ is denoted by $\gpl(H)$.
\end{deff}
For $g,h\in \gpl(H)$, we see that $gh\in\gpl(H)$, $\varepsilon_H(g)=1$ and $g\s_H(g)=1_H=\s_H(g)g$ (in particular, $g^{-1}=\s_H(g)\in\gpl(H)$),
where $\varepsilon_H$ (resp.~$1_H$) is the counit (resp.~unit element) of $H$.
Thus, $\gpl(H)$ forms an abstract group.

\subsection{Algebraic supergroups}\label{sec:alg-supergroup}
An \defnote{affine supergroup scheme} (\defnote{supergroup}, for short) over $\Bbbk$
is a representable functor $\G$ from the category of commutative superalgebras to the category of groups.
By Yoneda lemma, the representing object $\O(\G)$ of $\G$ forms a commutative Hopf superalgebra.
A supergroup $\G$ is said to be \defnote{algebraic} (resp.~\defnote{finite})
if $\O(\G)$ is finitely generated as a superalgebra (resp.~finite-dimensional).

For a supergroup $\G$,
we define its \defnote{even part} $\Gev$ as the restricted functor of $\G$ from the category of commutative algebras to the category of groups.
If we set $A:=\O(\G)$, then $\Gev$ is an (ordinary) affine group scheme represented by the quotient Hopf algebra $\overline{A}:=A/(A_\odd)$,
where $(A_\odd)$ is the super-ideal of $A$ generated by the odd part $A_\odd$ of $A$.
We denote $\overline{a}$ the image of $a\in A$ by the canonical quotient map $A\twoheadrightarrow \overline{A}$.
If $\G$ is algebraic, then so is $\Gev$.
An algebraic supergroup is said to be \defnote{connected} if its even part is connected, see \cite[Definition~8]{Mas12}.

\begin{ex} \label{ex:Gao}
We list some basic example of algebraic supergroups.
In the following, $R$ denotes a commutative superalgebra.
\begin{enumerate}
\item\label{ex:Gao(1)}
For positive integers $m$ and $n$,
we define the supergroup $\GL(m|n)$, called the \defnote{general linear supergroup}, by
\[
\GL(m|n)(R):=
\bigset{
\left(\begin{array}{c|c}g_{00} & g_{01}\\ \hline g_{10} & g_{11}\end{array}\right)}{
\begin{gathered}
g_{00}\in \GLL_m(R_\eve),g_{01}\in\Mat_{m,n}(R_\odd),\\
g_{10}\in \Mat_{n,m}(R_\odd),g_{11}\in\GLL_n(R_\eve)\\
\end{gathered}
},
\]
where $\GLL_m$ (resp.~$\Mat_{m,n}(R_\odd)$) denotes the general linear group scheme of size $m$
(resp.~the set of all $m\times n$ matrices whose entries are in $R_\odd$).
It is known that $\GL(m|n)$ is algebraic and its even part $\GL(m|n)_\ev$ is isomorphic to $\GLL_m\times \GLL_n$,
see \cite{BruKuj03,MarZub18,Zub06} for example.
\item\label{ex:Gao(2)}
For a positive integer $n$,
the following $\Q(n)$, called the \defnote{queer supergroup}, is a closed super-subgroup supergroup of $\GL(n|n)$.
\[
\Q(n)(R):=
\bigset{
\left(\begin{array}{c|c}g_{00} & g_{01}\\ \hline -g_{01} & g_{00}\end{array}\right)}{
g_{00}\in \GLL_n(R_\eve),g_{01}\in\Mat_{n,n}(R_\odd)}.
\]
The even part $\Q(n)_\ev$ of $\Q(n)$ is isomorphic to $\GLL_n$.
In \cite{Bru06,BruKle03}, modular representation theory of this supergroup is well-studied.
\item\label{ex:Gao(3)}
Let $\bigwedge(z)$ denote the exterior superalgebra of a one-dimensional vector space $\Bbbk z$ (see Example~\ref{ex:exterior}).
The corresponding algebraic supergroup of $\bigwedge(z)$ is denoted by $\Gao$,
called the \defnote{one-dimensional odd unipotent supergroup}, see \cite{GriZub16,MasZub17}.
By definition, we have $\Gao(R)=R_\odd$.
\qedd
\end{enumerate}
\end{ex}

Let $\G$ be a supergroup with representing object $\O(\G)$.
By a \defnote{left $\G$-supermodule} we mean a right $\O(\G)$-supercomodule.
A homomorphism of left $\G$-supermodules is just a right $\O(\G)$-supercomodule map.
For left $\G$-supermodules $V$ and $W$,
we set ${}_{\G}\Hom(V,W):=\Hom^{\O(\G)}(V,W)$ and ${}_{\G}\underline{\Hom}(V,W):=\underline{\Hom}^{\O(\G)}(V,W)$.
\medskip

A non-zero left $\G$-supermodule $L$ is said to be \defnote{simple} if $L$ has no non-trivial $\O(\G)$-super-subcomodule.
The parity change $\Pi$ acts on the set of isomorphism classes of simple left $\G$-supermodules $\Simple(\G)$ as a permutation of order two.
We let $\Simple_\Pi(\G)$ denote the set of $\Pi$-orbits in $\Simple(\G)$.

\subsection{Lie superalgebras and super-hyperalgebras}\label{sec:Lie(G)-hy(G)}
Let $\G$ be an algebraic supergroup.
Set $\m_{\G} := \Ker(\varepsilon_{\O(\G)})$, called the \defnote{augmentation super-ideal} of $\O(\G)$,
where $\varepsilon_{\O(\G)} : \O(\G) \to \Bbbk$ is the counit of $\O(\G)$.
Set $\Lie(\G) := (\m_{\G} / \m_{\G}^2 )^*$.
This naturally forms a Lie superalgebra (see \cite[Proposition~4.2]{MasShi17}),
which we call the \defnote{Lie superalgebra} of $\G$.
Since $\G$ is algebraic, $\Lie(\G)$ is finite-dimensional.
The even part $\Lie(\G)_\eve$ of $\Lie(\G)$ can be identified with the (ordinary) Lie algebra $\Lie(\Gev)$ of $\Gev$.

\begin{ex} \label{ex:Lie}
First, note that, $\Mat_{m|n}(\Bbbk)$ forms a Lie superalgebra
with Lie super-bracket $[X,Y]=XY-(-1)^{|X||Y|}YX$ for $X,Y\in \Mat_{m|n}(\Bbbk)$.
\begin{enumerate}
\item \label{ex:Lie(1)}
The Lie superalgebra of the general linear supergroup $\GL(m|n)$ is isomorphic to $\mathfrak{gl}(m|n):=\Mat_{m|n}(\Bbbk)$.
\item \label{ex:Lie(2)}
As a Lie super-subalgebra of $\mathfrak{gl}(n|n)$,
the Lie superalgebra of the queer supergroup $\Q(n)$ is isomorphic to
\[
\mathfrak{q}(n) :=
\bigset{
\left(\begin{array}{c|c}x_{00} & x_{01}\\ \hline x_{01} & x_{00}\end{array}\right)}{
x_{00},x_{01}\in \Mat_{n,n}(\Bbbk)}.
\]
This Lie superalgebra $\mathfrak{q}(n)$ is the so-called \defnote{queer superalgebra}.
\qedd
\end{enumerate}
\end{ex}

For any positive integer $n$, we regard $(\O(\G)/ \m_{\G}^n)^*$ as a super-subspace of $\O(\G)^*$ through
the dual of the canonical quotient map $\O(\G)\twoheadrightarrow \O(\G)/ \m_{\G}^n$.
As a sub-superspace of $\O(\G)^*$, we set
\[
\hy(\G) := \varinjlim_{n\geq1} (\O(\G) / \m_{\G}^n )^*.
\]
This $\hy(\G)$ forms a super-subalgebra of $\O(\G)^*$.
We call it the \defnote{super-hyperalgebra} of $\G$
(it is sometimes called the \defnote{super-distribution algebra} $\Dist(\G)$ of $\G$).
By definition, we see that $\hy(\G)=\O(\G)^*$ if $\G$ is infinitesimal, that is, $\G$ is finite and $\m_{\G}$ is nilpotent.
Since $\O(\G)/ \m_{\G}^n$ is finite-dimensional for any positive integer $n$,
one sees that $\hy(\G)$ has a structure of a cocommutative Hopf superalgebra such that the restriction
\[
\langle\;,\;\rangle \,: \, \hy(\G) \times \O(\G) \longrightarrow \Bbbk
\]
of the canonical pairing $\O(\G)^*\times \O(\G)\to \Bbbk$ is a Hopf pairing, see \cite[Lemma~5.1]{MasShi17}.
If $\G$ is connected, then the pairing induces an injection $\O(\G)\hookrightarrow \hy(\G)^*$.
In particular, the unit element of $\hy(\G)$ is given by the restriction of the counit $\varepsilon_{\O(\G)}:\O(\G)\to\Bbbk$ of $\O(\G)$.
\medskip

Masuoka showed the following {\it$\otimes$-split type theorem} for $\O(\G)$ and $\hy(\G)$,
see \cite[Theorem~4.5]{Mas05} for detail (see also \cite[Proposition~22]{Mas12} and \cite{MasShi18}).
\begin{thm} \label{thm:tensor-split}
For an algebraic supergroup $\G$,
there exists a counit (resp.~unit) preserving isomorphism
\[
\O(\G) \cong \O(\Gev) \otimes \bigwedge (\Lie(\G)_\odd^*)
\quad
\text{(resp.~}\hy(\G) \cong \hy(\Gev)\otimes \bigwedge(\Lie(\G)_\odd)\text{)}
\]
of (left $\O(\Gev)$-comodule) superalgebras (resp.~(left $\hy(\Gev)$-module) supercoalgebras).
\end{thm}

In the following, let $\s:\hy(\G)\to\hy(\G)$ denote the antipode of $\hy(\G)$ for simplicity.
Note that, $\s$ is the restriction of the dual $\s_{\O(\G)}^*$ of the antipode $\s_{\O(\G)}:\O(\G)\to \O(\G)$ of $\O(\G)$.
We set
\begin{equation} \label{eq:super-bracket}
[u,w] := \sum_{u,w} (-1)^{|u_{(2)}| |w_{(1)}|} u_{(1)} w_{(1)} \s(u_{(2)}) \s(w_{(2)}),
\quad u,w\in\hy(\G),
\end{equation}
called the \defnote{super-bracket} of $\hy(\G)$.
An element $X\in\hy(\G)$ is said to be \defnote{primitive} if the comultiplication of $X$ is given by $X\otimes 1 + 1\otimes X$,
where $1$ denotes the unit element of $\hy(\G)$.
For primitive elements $X,Y$ of $\hy(\G)$,
we have $[X,Y] = XY - (-1)^{|X| |Y|} YX$.
If we regard $\Lie(\G)$ as a super-subspace of $\hy(\G)$,
then this shows that $\Lie(\G)$ coincides with the set of all primitive elements in $\hy(\G)$.
\medskip

For a left $\G$-supermodule $V$,
we regard $V$ as a left $\hy(\G)$-supermodule by letting
\begin{equation} \label{eq:acthy}
u \acthy v := \sum_v (-1)^{|v_{(0)}| |u|}\,v_{(0)}\, \langle u, v_{(1)} \rangle,
\end{equation}
where $u\in\hy(\G)$, $v\in V$.
Suppose that $V$ is finite-dimensional.
Then the dual superspace $V^*$ of $V$ forms a right $\O(\G)$-supercomodule by using the antipode of $\O(\G)$.
The induced left $\hy(\G)$-supermodule structure on $V^*$ satisfies the following equation.
\begin{equation} \label{eq:acthydual}
(u \acthydual f) (v) = (-1)^{|u| |f|} f (\s(u) \acthy v),
\end{equation}
where $v\in V$, $f\in V^*$ and $u\in\hy(\G)$.

\subsection{Normal super-subgroups}\label{sec:normal}
Let $\G$ be an algebraic supergroup, in general.
Set $A:=\O(\G)$.
A subfunctor $\N$ of $\G$ is called a \defnote{closed super-subgroup}
if $\N$ is affine and the corresponding Hopf superalgebra $\O(\N)$ is isomorphic to
the quotient $A/I$ for some Hopf super-ideal $I$ of $A$.
This $\N$ is said to be \defnote{normal} if
(as an abstract group) $\N(R)$ is a normal subgroup of $\G(R)$ for all commutative superalgebra $R$.
The condition is equivalent to saying that the canonical quotient map $A\to A/I$ is \defnote{conormal} (see \cite[Definition~5.7]{Mas05}),
that is, $\coad_A(I)\subset A\otimes I$.
Here, $\coad_A$ denotes the (left) \defnote{coadjoint coaction} on $A$ given by
\[
\coad_{A} : A \longrightarrow A\otimes A;
\quad
a \longmapsto \sum_a (-1)^{|a_{(2)}| |a_{(3)}|} a_{(1)} \s_A(a_{(3)}) \otimes a_{(2)},
\]
where $\s_A$ is the antipode of $A$.
By definition, the even part $\N_\ev$ of a normal super-subgroup $\N$ of $\G$ is a normal subgroup of $\Gev$.

As the dual notion of (left) coadjoint coaction on $A$,
we define
\begin{equation} \label{eq:triact}
u\triact w := \sum_{u} (-1)^{|w| |u_{(2)}|} u_{(1)} w \s(u_{(2)}),
\quad u,w\in\hy(\G),
\end{equation}
called the (left) \defnote{adjoint action} on $\hy(\G)$,
where $\s$ is the antipode of $\hy(\G)$.
Note that, the super-bracket \eqref{eq:super-bracket} can be rewritten as $[u,w]=\sum_w (u\triact w_{(1)})\s(w_{(2)})$.

\begin{lemm} \label{prp:triact-prp}
For $u,u',w,w'\in\hy(\G)$,
we have
(1) $(uu')\triact w = u\triact(u'\triact w)$,
(2) $1\triact w = w$,
(3) $u\triact (ww') = \sum_{u}(-1)^{|u_{(2)}| |w|} (u_{(1)}\triact w)(u_{(2)}\triact w')$
and
(4) $u\triact 1 = \varepsilon(u) = u(1)$.
Here, $\varepsilon$ (resp.~$1$) denotes the counit (resp.~unit) of $\hy(\G)$.
\end{lemm}
\begin{proof}
It is straightforward to check (1), (2) and (4).
The following direct computation shows (3):
\begin{eqnarray*}
&& \sum_{u}(-1)^{|u_{(2)}| |w|} (u_{(1)}\triact w)(u_{(2)}\triact w') \\
&=& \sum_{u}(-1)^{(|u_{(2)}|+|u_{(3)}|+|u_{(4)}|)|w| + |u_{(4)}||w'|} u_{(1)} w \s(u_{(2)}) u_{(3)} w'\s(u_{(4)}) \\
&=& \sum_{u}(-1)^{|u_{(2)}| |w|+ |u_{(2)}||w'|} u_{(1)} w w'\s(u_{(2)}) = u\triact (ww').
\end{eqnarray*}
The second equation follows from the fact that $\varepsilon(x)=0$ for $|x|=1$ in general.
\end{proof}

Let $\N$ be a normal super-subgroup of $\G$.
By definition, the left coadjoint coaction on $\O(\G)$ induces a left $\O(\G)$-supercomodule structure on $\m_\N/\m_\N^2$,
where $\m_\N$ is the augmentation super-ideal of $\O(\N)$.
Since $\m_\N/\m_\N^2$ is finite-dimensional, its linear dual $\Lie(\N)$ has a left $\G$-supermodule structure.
Thus by \eqref{eq:acthydual}, we get a left $\hy(\G)$-supermodule structure on $\Lie(\N)=(\m_\N/\m_\N^2)^*$.

We regard $\Lie(\N)$ as a super-subspace of $\hy(\G)$ by the inclusion $\hy(\N)\subset \hy(\G)$.
Then one sees that the action of $\hy(\G)$ on $\Lie(\N)$ defined above is given by the adjoint action $\triact$, see \eqref{eq:triact}.
In particular, by restricting the action of $\hy(\G)$ to $\Lie(\G)$,
we see that $\Lie(\N)$ is a \defnote{Lie super-ideal} of $\Lie(\G)$,
that is,
$[X,N] \,(=X\triact N)\, \in\Lie(\N)$ for all $X\in\Lie(\G)$ and $N\in\Lie(\N)$.

A Hopf super-subalgebra $H$ of $\hy(\G)$ is said to be \defnote{normal} (see \cite[Theorem~3.10]{Mas05})
if $H$ is $\hy(\G)$-stable under the adjoint action $\triact$,
that is, $u\triact h \in H$ for all $u\in\hy(\G)$ and $h\in H$.
\begin{prop} \label{prp:hy(normal)}
If $\N$ is a normal super-subgroup of $\G$,
then $\hy(\N)\subset \hy(\G)$ is normal.
In particular, $\hy(\N)$ is closed under super-bracket of $\hy(\G)$,
that is, $[u,x]\in\hy(\N)$ for all $u\in\hy(\G)$ and $x\in\hy(\N)$.
\end{prop}
\begin{proof}
Let $A=\O(\G)$ and $I=\Ker(A\to\O(\N))$.
Under the natural pairing with $A$, we have $\hy(\N)=\{x\in\hy(\G)\mid x(I)=0\}$.
Let $\mathfrak m=\ker\varepsilon_A$ and $\overline{\mathfrak m}=\mathfrak m/I$.
Since the quotient map $A\to A/I$ sends $\mathfrak m^n$ onto $\overline{\mathfrak m}^{\,n}$, the asserted equality follows directly from the definition of the hyperalgebra.
Since $\N$ is normal, $\coad_A(I)\subset A\otimes I$.
By the duality between the coadjoint coaction and the adjoint action,
$u\triact x$ annihilates $I$ for all $u\in\hy(\G)$ and $x\in\hy(\N)$.
Thus, $\hy(\N)$ is normal in $\hy(\G)$.
The final assertion follows from
\[
[u,x]=\sum_x(u\triact x_{(1)})\s(x_{(2)}).
\]
This proves the claim.
\end{proof}

\subsection{Characters}\label{sec:char-cent}
Let $\Gm:=\GLL_1$ denote the one dimensional multiplicative group (scheme).
A \defnote{character} of a supergroup $\G$ is a group homomorphism from $\G$ to $\Gm$.
The set of all characters
\[
\X(\G):=\Hom(\G,\Gm)
\]
of $\G$, called the \defnote{character group} of $\G$, naturally forms an abstract group.
For $\chi\in\X(\G)$, we have a group homomorphism $\chi:\G(\O(\G))\to \Gm(\O(\G))$,
and hence we have a Hopf algebra homomorphism $\chi(\id_{\O(\G)}):\O(\Gm)\to \O(\G)$ by the Yoneda lemma.
If we realize $\O(\Gm)$ as the Laurent polynomial algebra $\Bbbk[X^{\pm1}]$ in the variable $X$ with coefficients in $\Bbbk$,
then it is easy to see that $\chi(\id_{\O(\G)})(X)\in \O(\G)_\eve$ is a group-like element.
In this way, we have an isomorphism
$\X(\G)\cong \gpl\big(\O(\G)\big)$
of abstract groups.

For each $\chi\in\X(\G)$,
we get the one-dimensional left $\G$-supermodule $\Bbbk^\chi$ so that
$\Bbbk^\chi=\Bbbk$ as a purely even superspace and the right $\O(\G)$-supercomodule structure is given by
\[
\Bbbk^\chi \longrightarrow \Bbbk^\chi\otimes \O(\G);
\quad
v\longmapsto v\otimes \chi.
\]
In other words, $g.v=\chi(g)v$ for all commutative superalgebra $R$ and $g\in \G(R), v\in \Bbbk^\chi$.
If there is no confusion, we sometimes simply denote $\Bbbk^\chi$ by $\chi$.
In this way, we get a one-to-one correspondence between $\X(\G)\cong \gpl(\O(\G))$ and the set of all equivalence classes of
one-dimensional (simple) left $\G$-supermodules under the parity change $\Pi$.
\begin{lemm}\label{prp:character}
The map $\X(\G)\to \X(\Gev);\;\chi\mapsto \chi|_{\Gev}$ is injective,
where $\chi|_{\Gev}$ denotes the restriction of $\chi:\G\to\Gm$ to $\Gev$.
\end{lemm}
\begin{proof}
Set $A:=\O(\G)$, $\Gamma:=\gpl(A)$ and $\Gamma':=\gpl(\overline{A})$.
Recall that, $\overline{A}=A/(A_\odd)=\O(\Gev)$.
Then the group algebra $\Bbbk\Gamma$ is a Hopf sub-superalgebra of $A$.
By \cite[Proposition~4.6(3)]{Mas05}, the inclusion $\Bbbk\Gamma\subset A$ induces an injection $\overline{\Bbbk\Gamma} \hookrightarrow \overline{A}$.
On the other hand, the quotient map $A\twoheadrightarrow \overline{A}$ induces a Hopf algebra homomorphism $\Bbbk\Gamma\to \Bbbk\Gamma'$.
Since $\Bbbk\Gamma =\overline{\Bbbk\Gamma}$, we see that $\Bbbk\Gamma \to \Bbbk\Gamma';\; a\mapsto \overline{a}$ is injective.
This proves the claim.
\end{proof}

\begin{ex} \label{ex:Ber}
We consider the case $\G=\GL(m|n)$.
Recall that $\Gev=\GLL_m\times \GLL_n$.
Let $R$ be a fixed superalgebra.
For
\[
g=\left(\begin{array}{c|c}g_{00} & g_{01} \\ \hline g_{10} & g_{11}\end{array}\right) \in \GL(m|n)(R),
\]
we set
$\det_{\eve}(g):=\det(g_{00})$, $\det_{\odd}(g):=\det(g_{11})$ and
\[
\Ber(g):=\det(g_{00}-g_{01}g_{11}^{-1}g_{10})\det(g_{11})^{-1}.
\]
This $\Ber(g)$ is called the \defnote{Berezinian} of $g$.
The Berezinian $\Ber$ belongs to $\X(\G)$.
If $\mathsf{char}(\Bbbk)=p>2$, then $\det_\epsilon^p\in\X(\G)$ for $\epsilon\in\mathz_2$.
Note that, $\Ber|_{\Gev} = (\det_{\eve}|_{\Gev}) \cdot (\det_{\odd}|_{\Gev})^{-1}$.
In \cite[Lemma~13.5]{Zub16}, Zubkov showed that the character group $\X(\G)$ of $\G=\GL(m|n)$ is generated by $\{\Ber,\det_{\odd}^{\tilde p}\}$,
where $\tilde p := \mathsf{char}(\Bbbk)\,(\neq2)$.
\qedd
\end{ex}

\section{Split Quasireductive Supergroups} \label{sec:Quasired}

\subsection{Split quasireductive supergroups} \label{sec:quasired}
Recall that,
a split and connected reductive $\mathz$-group $G_\mathz$ is a connected algebraic group (scheme) over $\mathz$ having a split maximal torus $T_\mathz$
such that the pair $(G_\mathz, T_\mathz)$ corresponds to a root datum (cf.~\cite{SGA3}).
See also \cite[Part~II, Chapter~1]{Jan03} and \cite[\S5.2]{Mil17}, for example.
It is known that $\mathcal{O}(G_\mathz)$ is free as a $\mathz$-module and $G_\mathz$ is infinitesimally flat.

\begin{deff}[{\cite[Definition~3.1]{Shi20}}]\label{def:Quasired}
An algebraic supergroup $\G_{\mathz}$ defined over $\mathz$ is said to be
\defnote{split quasireductive} if
its even part of $\G_{\mathz}$ is a split and connected reductive group over $\mathz$
and the odd part of $\m_{\G_\mathz}/ \m_{\G_\mathz}^2$ is finitely generated and free as a $\mathz$-module.
Here, $\m_{\G_\mathz}$ denotes the augmented ideal of $\O(\G_\mathz)$.
\end{deff}
Note that in \cite{Shi20,Shi21}, a split quasireductive supergroup is simply called a quasireductive supergroup.

In the following,
we fix a split quasireductive supergroup $\G_\mathz$ over $\mathz$ and a split maximal torus $T_\mathz$ of $(\G_\mathz)_{\ev}$.
Let $\G$ (resp. $T$) denote the base change of $\G_{\mathz}$ (resp. $T_\mathz$) to our base field $\Bbbk$,
that is, $\O(\G):=\O(\G_\mathz)\otimes_\mathz \Bbbk$.
By definition, $\G$ is connected.
We can identify $\X(T)$ with $\mathz^{\ell}$ ($\ell$ is the rank of $\Gev$) and we often write its group low additively with unit element $\mathbf{0}$.

\begin{ex} \label{ex:quasired}
We list some basic examples of split quasireductive supergroups.
\begin{enumerate}
\item
General linear supergroups $\GL(m|n)$.
\item
Queer supergroups $\Q(n)$.
\item
\defnote{Chevalley supergroups} of classical type, see \cite{FioGav12}.
For example, \defnote{special linear supergroups} $\SL(m|n)$ and \defnote{ortho-symplectic supergroups} $\SpO(m|n)$.
\item
\defnote{Periplectic supergroups} $\P(n)$ with $n\geq2$, see \cite{Shi20}.
For a superalgebra $R$, the supergroup is given by $\P(n)(R):=\{g\in \GL(n|n)(R) \mid {}^{\mathsf{st}}g\,J_n\,g = J_n\}$.
Here, we used the following notations.
\[
{\vphantom{\left(\begin{array}{c|c}g_{00}&g_{01}\\ g_{10}&g_{11}\end{array}\right)}}^{\mathsf{st}}
\left(\begin{array}{c|c}g_{00}&g_{01}\\\hline g_{10}&g_{11}\end{array}\right):=
\left(\begin{array}{c|c}{}^{\mathsf{t}}g_{00}&{}^{\mathsf{t}}g_{10}\\\hline -{}^{\mathsf{t}}g_{01}&{}^{\mathsf{t}}g_{11}\end{array}\right),
\quad
J_n:=\left(\begin{array}{c|c}O& I_n\\\hline I_n & O \end{array}\right),
\]
where ${}^{\mathsf{t}}g_{00}$ denotes the matrix transpose of $g_{00}$ and $I_n$ denotes the identity matrix of size $n$.
One sees that $\P(n)_\ev\cong \GLL_n$.
\qedd
\end{enumerate}
\end{ex}

As we have seen in Section~\ref{sec:Lie(G)-hy(G)},
for a left $\G$-supermodule $V$, we get a left $\hy(\G)$-supermodule structure on $V$.
It is easy to see that $V$ is locally finite and has a $T$-weight decomposition,
and hence $V$ becomes a locally finite left $\hy(\G)$-$T$-supermodule.
Here, we say that a left $\hy(\G)$-supermodule $V$ is \defnote{left $\hy(\G)$-$T$-supermodule}
if the restricted $\hy(T)$-supermodule structure on $V$ arises from some $T$-supermodule structure on it.
In this way, we get a functor from the category of left $\G$-supermodules
to the category of locally finite left $\hy(\G)$-$T$-supermodules.
\begin{thm}[{\cite[Theorem~5.8]{MasShi17}}] \label{thm:G=loc-fin-hy(G)-T}
The functor discussed above gives an equivalence between
the category of left $\G$-supermodules and
the category of locally finite left $\hy(\G)$-$T$-supermodules.
\end{thm}

\subsection{Root systems} \label{sec:root-sys}
Let $\g=\g_\eve\oplus\g_\odd$ be the Lie superalgebra $\Lie(\G)$ of $\G$.
As we have seen in Section~\ref{sec:normal} (for $\N=\G$),
the left coadjoint coaction of $\O(\G)$ induces the adjoint action of $\G$ on $\g$.
Restricting the action to $T$,
the Lie superalgebra $\g$ forms a left $T$-supermodule.
Since $T$ is a diagonalizable group scheme,
$\g$ decomposes into weight superspaces as follows:
\[
\g =
\bigoplus_{\alpha\in \X(T)} \g^\alpha =
(\bigoplus_{\alpha\in\X(T)} \g_{\eve}^\alpha)
\oplus
(\bigoplus_{\gamma\in\X(T)} \g_{\odd}^\gamma),
\]
where $\g^\alpha$ denotes the $\alpha$-weight super-subspace of $\g$.
By \cite[Part~I, 7.14]{Jan03}, we get
\[
\g^\alpha =
\{ X\in \g \mid u \triact X = \alpha(u) X \text{ for all } u\in\hy(T) \},
\]
where $\triact$ is the adjoint action \eqref{eq:triact}.
Here, we regard $\X(T)$ as a subset of $\O(T)$.
Let $\h:=\g^{\mathbf{0}}$ be the $\mathbf{0}$-weight super-subspace of $\g$ which forms a Lie super-subalgebra of $\g$.
Note that, the even part $\g_\eve$ of $\g$ coincides with the Lie algebra $\Lie(\Gev)$.
By definition, we see that $\h_\eve = \Lie(T)$.
For $\epsilon\in\mathz_2$, we set $\Deltaa_\epsilon := \{ \alpha\in\X(T) \mid \g_\epsilon^{\alpha} \not=0 \} \setminus \{\mathbf{0}\}$
and
\[
\Deltaa :=
\begin{cases}
\Deltaa_\eve \cup \Deltaa_\odd                     & \text{if } \h_{\odd}=0, \\
\Deltaa_\eve \cup \Deltaa_\odd \cup \{\mathbf{0}\} & \text{otherwise}.
\end{cases}
\]
We call $\Deltaa$ the \defnote{root system} of $\G$ with respect to $T$.
Note that, $\Deltaa_\eve$ is the root system of $\Gev$ with respect to $T$ in the usual sense.
Moreover, the quadruple $(\X(T),\Deltaa_\eve,\X(T)^\vee,\Deltaa_\eve^\vee)$ forms a root datum of the pair $(\Gev,T)$,
see \cite[Appendix~C]{Mil17}.
Let $\lambda_1,\dots,\lambda_\ell$ denote a basis of $\X(T)\cong\mathz^\ell=\bigoplus_{i=1}^\ell\mathz \lambda_i$,
where $\ell$ is the rank of $\Gev$.
\begin{ex} \label{ex:root-sys}
Here we list some examples of root systems.
\begin{enumerate}
\item \label{ex:root-sys(1)}
If $\G=\GL(m|n)$ with the standard maximal torus $T$ of $\Gev=\GLL_m\times\GLL_n$
(i.e., the subgroup of $\Gev$ consisting all diagonal matrices),
then $\X(T)\cong \bigoplus_{i=1}^{m+n}\mathz \lambda_i$ and
$\Deltaa=\Deltaa_\eve\sqcup\Deltaa_\odd=\{\lambda_i-\lambda_j \mid 1\leq i\neq j \leq m+n\}$
with $\Deltaa_\eve=\{\lambda_i-\lambda_j \mid 1\leq i\neq j \leq m\}\cup\{\lambda_i-\lambda_j \mid m+1\leq i\neq j \leq m+n\}$.
\item \label{ex:root-sys(2)}
If $\G=\Q(n)$ with the standard maximal torus $T$ of $\Gev=\GLL_n$,
then $\X(T)\cong \bigoplus_{i=1}^{n}\mathz \lambda_i$ and
$\Deltaa=\{\lambda_i-\lambda_j \mid 1\leq i\neq j\leq n\}\cup\{\mathbf{0}\}$
with $\Deltaa_\eve=\Deltaa_\odd$.
\item \label{ex:root-sys(3)}
If $\G=\P(n)$ with the standard maximal torus $T$ of $\Gev=\GLL_n$,
then $\X(T)\cong \bigoplus_{i=1}^{n}\mathz \lambda_i$ and
$\Deltaa_\eve=\{\lambda_i-\lambda_j \mid 1\leq i\neq j\leq n\}, \Deltaa_\odd=\{\pm(\lambda_i+\lambda_j),2\lambda_t \mid 1\leq i<j\leq n, 1\leq t\leq n\}$
with $\Deltaa=\Deltaa_\eve\sqcup\Deltaa_\odd$.
\item \label{ex:root-sys(4)}
Let $(\X,\R,\X^\vee,\R^\vee)$ be a root datum,
and let $F$ be a corresponding connected and split reductive group (defined over $\Bbbk$) with split maximal torus $T$.
Take group-like elements $g_1,\dots,g_n\in \gpl(\O(F))$.
By slightly modifying the algebraic supergroup $G_{g,x}$ given in \cite[Section~4]{MasZub17},
we consider the semidirect product
\[
\F^{\langle g_1,\dots,g_n\rangle}
:= F\ltimes (\Gao)^n
\]
such that
$\F^{\langle g_1,\dots,g_n\rangle}(R)=F(R)\times R_\odd^n$ as sets and the multiplication is
\[
\big(f, (x_i)_{1\leq i\leq n}\big).\big(k, (y_i)_{1\leq i\leq n}\big)
:= \big(fk, (k(g_i)x_i+y_i)_{1\leq i\leq n}\big)
\]
for $f,k\in F(R)$ and $(x_i)_{1\leq i\leq n},(y_i)_{1\leq i\leq n} \in R_\odd^n$,
where $R$ is a commutative superalgebra.
By definition, $(\F^{\langle g_1,\dots,g_n\rangle})_\ev = F$ and $\F^{\langle g_1,\dots,g_n\rangle}$ forms a split quasireductive supergroup.
The even part $\Deltaa_\eve$ of the root system $\Deltaa$ of $\F^{\langle g_1,\dots,g_n\rangle}$ with respect to $T$ is just $\R$.
Since $\gpl(\O(F))\hookrightarrow \gpl(\O(T)) = \X$,
we shall write $\chi_i:=g_i|_T$ for all $1\leq i\leq n$.
Then $\Deltaa_\odd=\{-\chi_i\mid 1\leq i\leq n,\ \chi_i\neq\mathbf{0}\}$.
\qedd
\end{enumerate}
\end{ex}

For each $\epsilon\in\mathz_2$, we set $\ell_\epsilon:=\dim(\h_\epsilon)$.
Note that, $\ell_\eve$ coincides with the rank $\ell$ of $\Gev$.
In \cite[Theorem~3.11]{Shi20}, Poincar\'e-Birkhoff-Witt (PBW) theorem for $\hy(\G)$ has been established.
It states that we can take a homogeneous basis
\[
\begin{gathered}
\{ X_\alpha\in\g_\eve^\alpha \mid \alpha\in\Deltaa_\eve \}
\cup
\{ H_i\in\h_\eve \mid 1\leq i\leq \ell_\eve \} \\
\cup
\{ Y_{(\gamma,j)}\in \g_\odd^\gamma \mid \gamma\in\Deltaa_\odd, 1\leq j\leq \dim(\g_\odd^\gamma) \}
\cup
\{ K_t\in\h_\odd \mid 1\leq t\leq \ell_\odd \}
\end{gathered}
\]
of $\g=\g_\eve\oplus\g_\odd$
so that
the set of all products of factors of the following type (taken in any fixed total order) forms a basis of $\hy(\G)$:
\[
H_{i}^{(m_i)}, \quad
X_{\alpha}^{(n_\alpha)}, \quad
K_{t}^{\epsilon_t},\quad
Y_{(\gamma,j)}^{\epsilon(\gamma,j)}
\]
with $n_\alpha,m_i\in\mathz_{\geq0}$, $\alpha\in\Deltaa_\eve$, $1\leq i\leq \ell_\eve$,
$\gamma\in\Deltaa_\odd$, $1\leq j \leq \dim(\g_\odd^\gamma)$, $1\leq t \leq \ell_\odd$
and
$\epsilon_t, \epsilon(\gamma,j) \in \{0,1\}$.
See also Theorem~\ref{thm:tensor-split}.
Here, we used the symbol of the ``divided powers'' $X_\alpha^{(n)}$ and $H_i^{(m)}$ for $X_\alpha$ and $H_i$.
For more detail, see \cite[\S3.4]{Shi20}.
In the following, to simplify the notation, we write $Y_\gamma:=Y_{(\gamma,1)}$ if $\dim(\g_\odd^\gamma)=1$ for $\gamma\in\Deltaa_\odd$.

One sees that $\hy(\G)$ is a cocommutative supercoalgebra of \defnote{Birkhoff-Witt type}
(for the non super-situation, see \cite[Proposition~3.3.5]{Tak74}).
In particular, if we denote the comultiplication of $\hy(\G)$ by $\Delta$, then we have
\begin{equation} \label{eq:BW-type}
\Delta(X_\alpha^{(n)}) = \sum_{i+j=n} X_\alpha^{(i)} \otimes X_\alpha^{(j)}
\quad\text{and}\quad
X_\alpha^{(n)} X_\alpha^{(m)} = \binom{n+m}{n}X_\alpha^{(m+n)}
\end{equation}
for $n,m\in\mathbb{N}\cup\{0\}$ and $\alpha\in\Deltaa_\eve$.
Here, $\binom{m+n}{n}$ denotes the binomial coefficient.

\subsection{Characters} \label{sec:char-cent2}
It is known that $\Gev$ is generated by
the split maximal torus $T$ and the \defnote{$\alpha$-root subgroups $U_\alpha$} of $\Gev$ for all $\alpha\in\Deltaa_\eve$,
see \cite[Theorem~21.11]{Mil17} for example.
Since each $U_\alpha$ is isomorphic to the one-dimensional additive group (scheme) $\Ga$,
we see that $\X(U_\alpha) \cong \gpl(\O(\Ga))$ is trivial,
and hence any character of $\Gev$ is trivial on $U_\alpha$.
In particular, the map $\X(\Gev)\to \X(T)$; $\chi\mapsto \chi|_T$ is injective.
\begin{rem} \label{rem:X_0(T)}
More precisely,
it is known (see \cite[Part~II, 1.18]{Jan03}) that
\[
\X(\Gev)\longrightarrow \X_0(T):=\{\lambda \in \X(T) \mid \langle \lambda, \alpha^\vee \rangle = 0 \text{ for all } \alpha\in\Deltaa_\eve \};
\quad \chi\longmapsto\chi|_T
\]
gives an isomorphism,
where $\alpha^\vee \in \X(T)^\vee=\Hom(\Gm,T))$ denotes the \defnote{dual root} corresponding to $\alpha$
and $\langle\;,\;\rangle$ denotes the perfect pairing $\X(T)\times \X(T)^\vee \to\mathbb{Z}$.
\qedd
\end{rem}

\begin{lemm} \label{prp:character2}
The map $\X(\G)\to \X(T);\;\chi\mapsto \chi|_{T}$ is injective.
More precisely, $\X(\G)\to \X_0(T);\;\chi\mapsto\chi|_{T}$ is injective.
\end{lemm}
\begin{proof}
By Lemma~\ref{prp:character} and Remark~\ref{rem:X_0(T)}, the claim follows immediately.
\end{proof}

\begin{ex}\label{ex:X(Q(n))}
We determine the character group $\X(\Q(n))$ of the queer supergroup $\Q(n)$.
One easily sees that $\X_0(T)=\{m(\lambda_1+\cdots+\lambda_n)\mid m\in\mathz\}$.
By Lemma~\ref{prp:character2}, any character $\chi$ restricts to $\det^m$ on $\Q(n)_\ev=\GLL_n$ for a unique $m\in\mathz$.
Let $K_i$ and $H_i$ be the standard odd and even diagonal elements of $\q(n)$, respectively, so that $[K_i,K_i]=2H_i$.
Since $K_i$ acts trivially on the one-dimensional supermodule defined by $\chi$, we have $m=d\chi(H_i)=0$ in $\Bbbk$.
In characteristic zero this forces $m=0$; in characteristic $p>2$ it forces $p\mid m$.
Conversely, $\det_\eve^p$ is a character in characteristic $p$ by Example~\ref{ex:Ber}.
Consequently,
\[
\X(\Q(n))=
\begin{cases}
\{1\} & \text{if }\mathsf{char}(\Bbbk)=0,\\
\{\det_\eve^{pm}\mid m\in\mathz\} & \text{if }\mathsf{char}(\Bbbk)=p>2.
\end{cases}
\]
Note that, $\Ber$ is trivial on $\Q(n)$.
\qedd
\end{ex}

\subsection{Frobenius kernels} \label{sec:Frob-ker}
In this subsection, we suppose that $\Bbbk$ is a perfect field of characteristic $p>2$ and fix a positive integer $r$.
Let $\G$ be an algebraic supergroup over $\Bbbk$, in general.

For a commutative superalgebra $R$,
we define a commutative superalgebra $R^{(r)}$ so that
$R^{(r)}=R$ as a super-ring and the scalar multiplication is given by $c.a = c^{p^{-r}} a$ for all $c\in \Bbbk$ and $a\in R$.
We define a supergroup $\G^{(r)}$
so that $\G^{(r)}(R) := \G(R^{(-r)})$,
and define a morphism $\Fr^r : \G \to \G^{(r)}$ of supergroups, called the \defnote{$r$-th Frobenius morphism},
as follows:
\[
\Fr^r(R) : \G(R) \longrightarrow \G^{(r)}(R);
\quad
g \longmapsto \big( \O(\G) \to R^{(-r)};\; a \mapsto g(a^{p^r}) \big).
\]
The kernel of the morphism $\Fr^r$ is called
the \defnote{$r$-th Frobenius kernel} of $\G$ which we denote by $\G_r$.

The supergroup $\G_r$ is represented by the quotient Hopf superalgebra $\O(\G)/I_r$, where
\[
I_r:=\big(a^{p^r}\mid a\in(\m_{\G})_\eve\big)
\]
is the super-ideal generated by the indicated powers and $\m_{\G}$ is the augmentation super-ideal of $\O(\G)$.
Since $a^{2}=0$ for all $a\in \O(\G)_\odd$,
we see that for a commutative superalgebra $R$ and $g\in \G(R)$,
the map $g':=(\Fr^r(R))(g):\O(\G)\to R^{(-r)}$ factors through the canonical quotient $\O(\G)\to\O(\Gev)$,
and hence we may identify $g'\in\Gevv(R^{(-r)})=\Gevv^{(r)}(R)$.
Thus, we may and do assume that $\Fr^r:\G\to \Gevv^{(r)}$.
By Theorem~\ref{thm:tensor-split} for $\G_r$, Masuoka \cite{Mas13} showed the following result:
\begin{prop}\label{prp:Gevr=Grev}
We have $(\Gev)_r=(\G_r)_\ev$ and $\Lie(\G_r)_\odd=\Lie(\G)_\odd$.
In particular, there exists a counit preserving isomorphism $\O(\G_r) \cong \O((\Gev)_r)\otimes \bigwedge (\Lie(\G)_\odd^*)$
of (left $\O((\Gev)_r)$-comodule) superalgebras.
\end{prop}
Therefore, $\Lie(\G_r)=\Lie(\G)$ and $\G_r$ is {\it infinitesimal},
that is,
$\G_r$ is finite and the augmentation super-ideal $\m_{\G_r}$ of $\O(\G_r)$ is nilpotent.
In particular, $\G_r$ is a finite normal super-subgroup of $\G$, and hence $\hy(\G_r)=\O(\G_r)^*$.
\medskip

For the twists below, including Example~\ref{ex:[r]}, let $\G$ be as in Section~\ref{sec:quasired}.
Using the fixed $\mathz$-form of $\Gev$ from Section~\ref{sec:quasired}, identify $\Gevv^{(r)}$ with $\Gev$ and write $F^r:\G\to\Gev$ for the resulting Frobenius morphism.
It has kernel $\G_r$ and induces an isomorphism $\G/\G_r\cong\Gev$.
For a left $\Gev$-supermodule $V$, define $V^{[r]}$ to be its pullback along $F^r$, with the same underlying superspace; an ordinary module is regarded as purely even.
For a left $\G$-supermodule $M$ on which $\G_r$ acts trivially, define $M^{[-r]}$ through the isomorphism $\G/\G_r\cong\Gev$ (see \cite{MasZub11} for quotient sheaves).
Then $(M^{[-r]})^{[r]}=M$ as a $\G$-supermodule.
\begin{ex} \label{ex:[r]}
Let $M$ be a left $\G$-supermodule.
For the \defnote{$\G_r$-fixed point super-subspace} $M^{\G_r}$ of $M$,
we can consider $(M^{\G_r})^{[-r]}$.
We naturally regard $M$ as a left $\G_r$-supermodule via the inclusion $\G_r\subset \G$.
For a finite dimensional left $\G$-supermodule $M'$,
we can make ${}_{\G_r}\underline{\Hom}(M',M)$ into a left $\G$-supermodule by the conjugate action.
As a left $\hy(\G)$-supermodule, the induced action is given by
\[
(u.f)(v) := \sum_u (-1)^{|f||u_{(2)}|}\, u_{(1)} f(\s(u_{(2)}) v),
\]
where $f\in {}_{\G_r}\underline{\Hom}(M',M)$, $u\in\hy(\G)$ and $v\in M'$.
Here, $\s$ denotes the antipode of $\hy(\G)$.
Since $M^{\G_r}$ can be identified with ${}_{\G_r}\underline{\Hom}(\Bbbk,M)$,
we can also consider ${}_{\G_r}\underline{\Hom}(M',M)^{[-r]}$.
Note that, the ``evaluation map''
\[
\varphi: {}_{\G_r}{\Hom}(M',M) \otimes M'\longrightarrow M;
\quad
f\otimes v\mapsto f(v)
\]
is a morphism of superspaces,
since ${}_{\G_r}{\Hom}(M',M)={}_{\G_r}\underline{\Hom}(M',M)_\eve$ consists of parity preserving morphisms.
Moreover, we get
\[
\varphi(u.(f\otimes v))
= \sum_{u} (u_{(1)}.f)(u_{(2)}v)
= \sum_{u} u_{(1)}f(\s(u_{(2)})u_{(3)}v)
=u \varphi(f\otimes v)
\]
for each $u\in\hy(\G)$, $f\in {}_{\G_r}\Hom(M',M)$ and $v\in M'$.
This shows that $\varphi$ is actually a $\G$-supermodule homomorphism.
\qedd
\end{ex}

Set $\g:=\Lie(\G)$.
Let $r$ be a fixed positive integer.
Set $n_\epsilon := \dim(\g_\epsilon)$ for $\epsilon\in\mathz_2$.
Choose homogeneous elements $f_1,\dots,f_{n_\eve+n_\odd}\in\m_{\G}$ whose residue classes form a basis of
$\m_{\G}/\m_{\G}^2\cong\g^*$, with $f_1,\dots,f_{n_\eve}$ even and the remaining elements odd.
Since $\Gev$ is reduced, the images of
$\{f_1^{a_1}\cdots f_{n_\eve}^{a_{n_\eve}} \mid 0\leq a_1,\dots,a_{n_\eve}\leq p^r-1 \}$
form a basis of $\O((\Gev)_r)$, see \cite[Part~I, 9.6]{Jan03}.
We use the same symbols for the images of the $f_i$ in the relevant quotient coordinate algebras.
Thus, by Proposition~\ref{prp:Gevr=Grev}, the set
\begin{equation} \label{eq:basis_of_O(G_r)}
\bigset{
f_1^{a_1}\cdots f_{n_\eve}^{a_{n_\eve}}\cdot f_{n_\eve+1}^{\epsilon_1} \cdots f_{n_\eve+n_\odd}^{\epsilon_{n_\odd}}}{
\begin{gathered}
0\leq a_1,\dots,a_{n_\eve}\leq p^r-1,\\
\epsilon_1,\dots,\epsilon_{n_\odd}\in\{0,1\}
\end{gathered}}
\end{equation}
forms a basis of $\O(\G_r)$.
In particular, we have
\[
\dim(\O(\G_r)) = p^{r n_\eve}\cdot 2^{n_\odd}.
\]

\begin{ex} \label{ex:Gao_r}
Recall that, $\Gao$ is the one-dimensional odd unipotent supergroup
with $\O(\Gao)=\Bbbk[z]/(z^2)$, see Example~\ref{ex:Gao}\eqref{ex:Gao(3)}.
Then for a commutative superalgebra $R$, we have
\[
\Fr^r(R):\Gao(R)\longrightarrow (\Gao)^{(r)}(R);
\quad
g\longmapsto(z\mapsto z^{p^r}\mapsto g(z^{p^r})).
\]
Since $p>2$ and $z^2=0$,
we conclude that the $r$-th Frobenius kernel $(\Gao)_r$ of $\Gao$ coincides with $\Gao$.
For the supergroup $\F^{\langle g_1,\dots,g_n\rangle}=F\ltimes(\Gao)^n$ defined in Example~\ref{ex:root-sys}\eqref{ex:root-sys(4)},
if the group-like elements $g_1,\dots,g_n$ are trivial,
then $\F^{\langle 1,\dots,1\rangle}=F\times (\Gao)^n$ and $(\F^{\langle 1,\dots,1\rangle})_r = F_r\times (\Gao)^n$,
where $F_r$ denotes the $r$-th Frobenius kernel of $F$.
\qedd
\end{ex}

In the following, we regard $\hy(\G_r)$ as a Hopf super-subalgebra of $\hy(\G)$ via the inclusion $\G_r\subset \G$.
The following is a PBW type theorem for the $r$-th Frobenius kernel $\G_r$ of $\G$.
\begin{thm}\label{prp:PBW-Gr}
For any total order on the homogeneous basis of $\g=\g_\eve\oplus\g_\odd$,
the set of all products of factors of type
\[
H_{i}^{(m_i)}, \quad
X_{\alpha}^{(n_\alpha)}, \quad
K_{t}^{\epsilon_t},\quad
Y_{(\gamma,j)}^{\epsilon(\gamma,j)}
\]
($0\leq n_\alpha,m_i\leq p^r-1$, $\alpha\in\Deltaa_\eve$, $1\leq i\leq \ell_\eve$,
$\gamma\in\Deltaa_\odd$, $1\leq j \leq \dim(\g_\odd^\gamma)$, $1\leq t \leq \ell_\odd$
and
$\epsilon_t, \epsilon(\gamma,j) \in \{0,1\}$),
taken in $\hy(\G)$
with respect to the order, form a basis of $\hy(\G_r)$.
\end{thm}
\begin{proof}
By Theorem~\ref{thm:tensor-split} for $\hy(\G_r)$ and Proposition~\ref{prp:Gevr=Grev},
we have an isomorphism $\hy(\G_r) \cong \hy((\Gev)_r) \otimes \bigwedge(\g_\odd)$ of (left $\hy((\Gev)_r)$-module) supercoalgebras.
On the other hand, since $\Gev$ is split reductive,
the set of all products (taken in the fixed order) of factors of type
$H_{i}^{(m_i)},X_{\alpha}^{(n_\alpha)}$
($0\leq n_\alpha,m_i\leq p^r-1$, $\alpha\in\Deltaa_\eve$, $1\leq i\leq \ell_\eve$)
form a basis of $\hy((\Gev)_r)$, see \cite[Part~II, Lemma~3.3]{Jan03}.
The proof is done.
\end{proof}

In particular,
it follows that $\hy(\G)$ is generated by $\hy(\Gev)$ and $\hy(\G_r)$
as a superalgebra.

\section{Unimodularity of Algebraic Supergroups} \label{sec:Unimo}
In this section,
we discuss the unimodularity of Frobenius kernels of split quasireductive supergroups.

\subsection{(Co)integrals on Hopf superalgebras} \label{sec:(co)int}
Let $H$ be a Hopf superalgebra with unit $1_H$ and counit $\varepsilon_H$, in general.
A \defnote{left cointegral on $H$} is an element $\phi\in H^*$ satisfying
\[
f * \phi = f(1_H) \phi
\]
for all $f\in H^*$.
Here, $f*\phi:H\to \Bbbk$ denotes the convolution product of $f$ and $\phi$,
that is, $(f*\phi)(h)=\sum_h (-1)^{|h_{(1)}||\phi|} f(h_{(1)}) \phi(h_{(2)})$ for $h\in H$.
In other words,
a left cointegral is an element in the space $\tint_\LL^H := \underline{\Hom}^{H}(H,\Bbbk)$,
where $\Bbbk$ is regarded as a trivial left $H$-supercomodule.
The notion of a \defnote{right cointegral on $H$} and the symbol $\tint_\RR^H$ are defined analogously.
Using the {\it bosonization technique} (see \cite[Section~10]{MasZub11} for example), we have the following:
\begin{prop}[{\cite[Corollary~3.2]{MasShiShi22}}] \label{prp:finite-integral}
Both of $\dim(\tint_\RR^H)$ and $\dim(\tint_\LL^H)$ are less than or equal to $1$,
that is,
a non-zero left or right cointegral on $H$ is unique up to scalar multiplication if it exists.
Moreover, such an element is homogeneous.
\end{prop}

\begin{deff} \label{def:unimo}
We say that $H$ is \defnote{unimodular} if $\tint_\LL^H=\tint_\RR^H\not=0$,
that is,
there exists a non-zero two-sided (i.e, left and right) cointegral on $H$.
\end{deff}

Suppose that $H$ is finite-dimensional.
An element $t\in H$ is called a \defnote{left} (resp.~\defnote{right}) \defnote{integral in $H$} if it satisfies
$h t = \varepsilon_H(h) t$ (resp.~$t h = \varepsilon_H(h) t$) for all $h\in H$.
The space of all left (resp.~right) integrals in $H$ is denoted by $\ttint_H^\LL$ (resp.~$\ttint_H^\RR$).

In general, it is known that any finite dimensional Hopf algebra has both non-zero left/right integral.
By this fact and the dual result of \cite[Proposition~3.1]{MasShiShi22},
we have $\dim(\ttint_H^\LL) = \dim(\ttint_H^\RR)=1$ and $\s_H (\ttint_H^\LL) = \ttint_H^\RR$,
where $\s_H : H\to H$ is the antipode of $H$.
As in the non super-situation (see \cite[Chapter~10]{Rad12}), one easily sees that the following holds:
\begin{prop}\label{prp:dist-grp-like}
There uniquely exists $\dist_H \in \gpl(H^*)$ such that
$t h = \langle \dist_H, h \rangle t$
for all $h\in H \text{ and } t\in \ttint_H^\LL$.
\end{prop}
The element $\dist_H$ is the so-called \defnote{distinguished group-like element} for $H$.

\subsection{Integrals for supergroups} \label{sec:int-for-supergroup}
Let $\G$ be an algebraic supergroup, in general.
We say that $\G$ has a \defnote{left} (resp.~\defnote{right}) \defnote{integral for $\G$}
if there exists a non-zero left (resp.~right) cointegral on $\O(\G)$.
Also, we say that $\G$ is \defnote{unimodular} if $\O(\G)$ is unimodular (see Definition~\ref{def:unimo}).
\begin{thm}[{\cite[Theorem~3.7]{MasShiShi22}}] \label{prp:G-int=Gev-int}
$\G$ has a left (resp.~right) integral if and only if $\Gev$ does.
\end{thm}

Assume for a moment that $\mathsf{char}(\Bbbk)=0$.
Let $F$ be an algebraic group over $\Bbbk$.
Then by Sullivan's theorem (\cite{Sul72}),
$F$ has a left (or right) integral if and only if $F$ is linearly reductive.
In particular, in this case, $F$ is automatically unimodular.
However, in our super-situation,
the existence of an integral does not imply its unimodularity (see Theorem~\ref{prp:unimo-char=0} below).

By Theorem~\ref{prp:G-int=Gev-int} (and Sullivan's theorem again),
we note that for a connected and algebraic supergroup $\G$ defined over a filed of characteristic zero,
$\G$ has a left (or right) integral if and only if $\Gev$ is reductive.

\begin{thm} \label{prp:unimo-char=0}
Assume that $\mathsf{char}(\Bbbk)=0$ and $\G$ is a split quasireductive supergroup.
Then $\G$ is unimodular if and only if $\sum_{\gamma\in\Deltaa_{\odd}}\dim(\g_\odd^\gamma) \gamma=\mathbf{0}$ in $\X(T)$.
\end{thm}
\begin{proof}
Let $\ad':\g_\eve\to \End(\g_\odd)$ be the restriction of the adjoint representation of $\g$.
By \cite[Proposition~3.17]{MasShiShi22}, $\G$ is unimodular if and only if the algebra map $\chi_\G:\mathcal{U}(\g_\eve)\to\Bbbk$ determined by
\[
\chi_\G(X)=\mathsf{tr}(\ad'(X))
\quad\text{for all }X\in\g_\eve
\]
is trivial.
Let $D\in\X(\Gev)$ be the determinant of the adjoint action of $\Gev$ on $\g_\odd$.
Its differential satisfies $dD(X)=\mathsf{tr}(\ad'(X))$ for all $X\in\g_\eve$.
Thus, $\chi_\G$ is the algebra map induced by $D$ on $\mathcal{U}(\g_\eve)$.
Since $\hy(\Gev)=\mathcal{U}(\g_\eve)$ and the connectedness of $\Gev$ gives an injection $\O(\Gev)\hookrightarrow\hy(\Gev)^*$, $\chi_\G$ is trivial if and only if $D$ is trivial.
By Lemma~\ref{prp:character2} applied to $\Gev$, this is equivalent to the triviality of $D|_T$.
The $T$-weight decomposition $\g_\odd=\h_\odd\oplus\bigoplus_{\gamma\in\Deltaa_\odd}\g_\odd^\gamma$ gives
\[
D(t)=\prod_{\gamma\in\Deltaa_\odd}\gamma(t)^{\dim(\g_\odd^\gamma)}
\]
for all $t\in T(R)$, where $R$ is a commutative algebra.
Hence, $D|_T$ is trivial if and only if $\sum_{\gamma\in\Deltaa_\odd}\dim(\g_\odd^\gamma)\gamma=\mathbf{0}$ in $\X(T)$.
This proves the claim.
\end{proof}

\begin{cor} \label{cor:quasi-unimo}
Assume that $\mathsf{char}(\Bbbk)=0$.
Then $\GL(m|n)$, $\Q(n)$ and Chevalley supergroups of classical type are unimodular.
\end{cor}
\begin{proof}
For $\GL(m|n)$, $\Q(n)$ and Chevalley supergroups of classical type other than $P$, we can define an ``order'' on $\Deltaa_\odd$ satisfying the following properties, as in Section~\ref{sec:simple-G-supermods} and Example~\ref{ex:bases}:
\begin{equation}\label{eq:condition}
\begin{gathered}
\Deltaa_\odd=\Deltaa_\odd^+\sqcup\Deltaa_\odd^- \,\,
\text{(disjoint union)},\quad \Deltaa_\odd^+=-\Deltaa_\odd^-,\\
\text{and}\quad \dim(\g_\odd^\gamma)=\dim(\g_\odd^{-\gamma})\quad \text{for all } \gamma\in\Deltaa_\odd^+
\end{gathered}
\end{equation}
Thus, in these cases, we have $\sum_{\gamma\in\Deltaa_{\odd}}\dim(\g_\odd^\gamma)\gamma=\mathbf{0}$.
If $\G$ is a Chevalley supergroup of type $P$, then $\Gev$ is semisimple.
Hence, the determinant of its adjoint action on $\g_\odd$ is a trivial character of $\Gev$.
As in the proof of Theorem~\ref{prp:unimo-char=0}, its restriction to $T$ corresponds to $\sum_{\gamma\in\Deltaa_\odd}\dim(\g_\odd^\gamma)\gamma$ in $\X(T)$, so this sum is also zero.
By Theorem~\ref{prp:unimo-char=0}, we are done.
\end{proof}

The condition~\eqref{eq:condition} in the cases other than type $P$ holds without assuming that $\mathsf{char}(\Bbbk)=0$.
For type $P$, the even group is semisimple in positive characteristic as well, so the determinant character is still trivial.

\begin{ex} \label{ex:non-unimo_0}
Assume that $\mathsf{char}(\Bbbk)=0$.
\begin{enumerate}
\item \label{ex:non-unimo_0(1)}
Suppose that $\G=\P(n)$ with $n\geq2$.
Then by Example~\ref{ex:root-sys}\eqref{ex:root-sys(3)},
we have $\sum_{\gamma\in\Deltaa_{\odd}}\dim(\g_\odd^\gamma) \gamma = 2\sum_{t=1}^n\lambda_t \neq \mathbf{0}$.
Thus, $\P(n)$ is non-unimodular.
\item \label{ex:non-unimo_0(2)}
We consider the following closed supergroup $\G$ of $\GL(3|3)$.
\[
\G(R) := \{
\left(\begin{array}{ccc|ccc}
h & 0 & x & 0 & 0 & 0 \\
0 & 1 & 0 & a & 0 & b \\
y & 0 & k & 0 & 0 & 0 \\ \hline
0 & 0 & 0 & h & 0 & x \\
a & 0 & b & 0 & 1 & 0 \\
0 & 0 & 0 & y & 0 & k
\end{array}\right)
\in \GL(3|3)(R)
\},
\]
where $R$ is a superalgebra.
Since $\Gev\cong \GLL_2$, this is split quasireductive.
If we take $T$ as diagonal matrices in $\G$, then root system of $\G$ with respect to $T$ is given by
$\Deltaa_\eve=\{\pm(\lambda_1-\lambda_3)\}$ and $\Deltaa_\odd=\{-\lambda_1,-\lambda_3\}$.
Thus, $\sum_{\gamma\in\Deltaa_{\odd}}\dim(\g_\odd^\gamma) \gamma=-\lambda_1-\lambda_3\neq\mathbf{0}$,
and hence this $\G$ is non-unimodular.
\item \label{ex:non-unimo_0(3)}
For the split quasireductive supergroup $\F^{\langle g_1,\dots,g_n\rangle}=F\ltimes (\Gao)^n$ discussed in Example~\ref{ex:root-sys}\eqref{ex:root-sys(4)},
we have seen that $\Deltaa_\odd=\{-\chi_1,\dots,-\chi_n\}\setminus\{\mathbf{0}\}$.
Set $m:=\#\Deltaa_\odd$.
If we write $\Deltaa_\odd = \{-\chi_{i_1},\dots,-\chi_{i_m}\}$ and set
\[
d_j:=\#\{i\in\{1,\dots,n\}\mid\chi_i=\chi_{i_j}\}=\dim(\g_\odd^{-\chi_{i_j}}),
\]
then $\F^{\langle g_1,\dots,g_n\rangle}$ is unimodular if and only if $\sum_{j=1}^m d_j \chi_{i_j} = \mathbf{0}$.
\qedd
\end{enumerate}
\end{ex}

\subsection{Integrals for finite normal super-subgroups} \label{sec:fin-normal}
Again, we suppose $\Bbbk$ is a field of characteristic different from $2$.
Let $\G$ be an algebraic supergroup over $\Bbbk$,
and let $\N$ be a finite and normal super-subgroup of $\G$.
Set $A:=\O(\G)$ and $B:=\O(\N)$ for simplicity.
For $a\in A$,
we denote by $\overline{a}^B \in B$ the image of $a$ via the canonical Hopf quotient map $A\twoheadrightarrow B$
corresponding to the inclusion $\N\subset \G$.

Since $\N$ is normal,
the left adjoint action $\Ad$ of $\G$ on $\N$ makes $B$ into a Hopf superalgebra object in the category of left $A$-supercomodules.
Explicitly,
\[
\coad_B : B \longrightarrow A\otimes B;
\quad
\overline{a}^B \longmapsto \sum_a (-1)^{|a_{(2)}| |a_{(3)}|} a_{(1)} \s_A (a_{(3)}) \otimes \overline{a_{(2)}}^B,
\]
where $\s_A$ is the antipode of $A$.
Taking the linear dual,
$B^*$ forms a Hopf superalgebra object in the category of right $A$-supercomodules with the dual supercomodule structure map
$\coad^*_B : B^* \to B^*\otimes A$ of $\coad_B$.

Since $B$ is finite-dimensional, the space $\ttint_{B^*}^\LL$ of left integrals in $B^*$ is one dimensional, see Section~\ref{sec:(co)int}.
In the following, we take and fix a $\Bbbk$-base $\phi$ of the space $\ttint_{B^*}^\LL$ of left integrals in $B^*$,
that is, $\ttint_{B^*}^\LL = \Bbbk\,\phi$.
Note that, $\phi$ is homogeneous, that is, purely even or odd.

\begin{lemm} \label{prp:dist-grp}
The space $\ttint_{B^*}^\LL$ forms an $A$-super-subcomodule of $B^*$.
In particular, there uniquely exists $\chi \in \gpl(A)$ such that $\coad^*_B(\phi) = \phi\otimes \chi$.
\end{lemm}
\begin{proof}
We denote by $\Phi : \G \to \underline{\Aut}(B^*)$ the left $\G$-supermodule structure map on $B^*$
corresponding to $\coad_B^* : B^* \to B^* \otimes A$.
To prove the claim we show that $\ttint_{B^*}^\LL$ is stable under the action of ${}^g(-) := \Phi_R(g)(-)$
for all commutative superalgebra $R$ and $g\in \G(R)$,
where $\Phi_R : \G(R) \to \underline{\Aut}_R(B^*\otimes R)$
and $\underline{\Aut}_R(B^*\otimes R):=\underline{\End}_R(B^*\otimes R)^\times$.

We fix $f\in B^*$.
Since $B^*$ is a Hopf superalgebra object in the category of right $A$-supercomodules,
we have
\[
(f\otimes1_R) * {}^g(\phi\otimes1_R) = {}^g \big({}^{g^{-1}} (f\otimes1_R) * (\phi\otimes1_R)\big),
\]
where $1_R$ is the unit element of $R$.
On the other hand, since $\phi$ is a left integral in $B^*$,
we have
\[
{}^{g^{-1}} (f\otimes1_R) * (\phi\otimes1_R)
= \varepsilon_{B^*\otimes R}\big( {}^{g^{-1}} (f\otimes1_R) \big) (\phi\otimes1_R).
\]
By definition, we get $\varepsilon_{B^*\otimes R}\big( {}^{g^{-1}} (f\otimes1_R) \big) = \varepsilon_{B^*}(f)\otimes1_R$.
Thus, we conclude that ${}^g (\phi\otimes1_R) \in \ttint_{B^*}^\LL\otimes R$.
\end{proof}

We may identify the dual superspace $B^{**}$ of $B^*$ with $B$,
since $B$ is finite-dimensional.
Through this identification,
there uniquely exists $\dist_{B^*} \in \gpl(B)$ (i.e., the distinguished group-like element) such that
\[
\phi * f = \langle f, \dist_{B^*} \rangle \phi
\quad \text{for all} \quad
f\in B^*,
\]
see Proposition~\ref{prp:dist-grp-like}.
Note that, $\dist_{B^*}$ is an element of the even part $B_\eve$ of $B$.

The left $A^*$-supermodule structure on $B^*$ induced from $\coad^*_B:B^*\to B^*\otimes A$ is given by
\begin{equation} \label{eq:adjoint-act}
h \hits f = \sum_f (-1)^{|h||f_{(0)}|} f_{(0)} \langle h, f_{(1)} \rangle
\quad \text{for all} \quad
h\in A^* \text{ and } f\in B^*,
\end{equation}
where we write $\coad^*_B(f) = \sum_f f_{(0)} \otimes f_{(1)}$.
By restricting the action to $B^*\,(\subset A^*)$,
we get the adjoint action $k \hits f = \sum_k (-1)^{|f||k_{(2)}|} k_{(1)} * f * \s_{B^*}(k_{(2)})$ for all $k,f\in B^*$,
where $\s_{B^*}$ is the antipode of $B^*$.

\begin{prop} \label{prp:beta=alphaS}
$\overline{\chi}^B$ coincides with the inverse $(\dist_{B^*})^{-1}$ of the distinguished group-like element $\dist_{B^*}\in\gpl(B)$ of $B^*$.
\end{prop}
\begin{proof}
We fix $k\in B^*$.
Since $\phi \in \ttint_{B^*}^\LL$ is purely even/odd and $\dist_{B^*}\in B_\eve$, we have
\begin{eqnarray*}
k \hits \phi
&=& \sum_k (-1)^{|k_{(2)}|} k_{(1)} * \phi * \s_{B^*}(k_{(2)}) \\
&=& \sum_k (-1)^{|k_{(2)}|} \varepsilon_{B^*}(k_{(1)}) \langle S_{B^*}(k_{(2)}), \dist_{B^*} \rangle \phi \\
&=& \langle S_{B^*}(k), \dist_{B^*} \rangle \phi
= \langle k, (\dist_{B^*})^{-1} \rangle \phi.
\end{eqnarray*}
On the other hand, we calculate the action $k \hits \phi$ directly.
Since we know $\coad^*_B(\phi) = \phi \otimes \chi$ by Lemma~\ref{prp:dist-grp},
we get
\[
k \hits \phi
= (-1)^{|k| |\chi|} \langle k, \overline{\chi}^B \rangle \phi
= \langle k, \overline{\chi}^B \rangle \phi
\]
by \eqref{eq:adjoint-act}.
The last equation holds since $\chi\in A_\eve$.
Combining these results, we get $\langle k, \overline{\chi}^B - (\dist_{B^*})^{-1} \rangle = 0$ for all $k\in B^*$.
This proves the claim.
\end{proof}

If we identify $\gpl(A)$ with $\X(\G)$,
then $\overline{\chi}^B\in\gpl(B)$ is identified with the restriction $\chi|_\N\in\X(\N)$.
Using this, we can rephrase Proposition~\ref{prp:beta=alphaS} as follows:
\begin{thm} \label{prp:criterion-unimo}
The restriction $\chi|_\N$ is trivial if and only if $\N$ is unimodular.
In particular, $\N$ is unimodular if $\chi$ is trivial.
\end{thm}

\begin{rem}\label{rem:non-super-unimo}
In the non super-situation,
Theorem~\ref{prp:criterion-unimo} tells us that
for a connected and split reductive group $F$, any finite and normal subgroup $K$ of $F$ is unimodular.
In particular, all Frobenius kernels of $F$ are unimodular.
We give a proof of this fact.
The adjoint action $\Ad:F\to \Aut(K)$; $f\mapsto(k\mapsto fkf^{-1})$ factors through the quotient $F/\Z(F)$,
where $\Z(F)$ is the center of $F$.
Thus, the corresponding coaction $\O(K)^*\to \O(K)^*\otimes \O(F)$ factors through $\O(K)^*\otimes \O(F/\Z(F))$:
\[
\begin{xy}
(0,0)*++{\O(K)^*}="1",
(50,0)*++{\O(K)^*\otimes \O(F)}="2",
(50,-10)*++{\O(K)^*\otimes \O(F/\Z(F)).}="3",
{"1"\SelectTips{cm}{}\ar@{->}"2"},
{"1"\SelectTips{cm}{}\ar@{-->}"3"},
{"3"\SelectTips{cm}{}\ar@{^(->}"2"}
\end{xy}
\]
Note that, we regard $\O(F/\Z(F))$ as a Hopf subalgebra of $\O(F)$
via the canonical quotient $F\twoheadrightarrow F/\Z(F)$.
Thus, the group-like element $\chi$ is in $\O(F/\Z(F))$.
On the other hand, since $F$ is connected and reductive, the quotient $F/\Z(F)$
coincides with its derived group,
see \cite[Chapter~21]{Mil17} for example.
Thus, there is no non-trivial group-like element in $\O(F/\Z(F))$,
and hence $\chi$ must be trivial.
Then by Theorem~\ref{prp:criterion-unimo}, $K$ is unimodular.
\qedd
\end{rem}

However, in our super-situation,
the proof in Remark~\ref{rem:non-super-unimo} does not work.
One of the reasons is that $(\G/\Z(\G))_\ev = \Gev/\Z(\G)_\ev$ (by Masuoka and Zubkov \cite{MasZub11}) is not isomorphic to $\Gev/\Z(\G_\ev)$, in general.
For example, if we take $\G=\GL(m|n)$,
then $\Z(\Gev)\cong\Z(\GLL_m\times \GLL_n)\cong\Gm\times \Gm$,
while $\Z(\G)=\Z(\G)_\ev \cong \Gm$.

\subsection{Unimodularity of Frobenius kernels} \label{sec:unimod-Frob-ker}
We suppose that the base field $\Bbbk$ is a perfect field of characteristic $p>2$.

In \cite[Corollary~7.2]{ZubMar16}, it is proved that
all Frobenius kernels of the general linear supergroup $\GL(m|n)$ are unimodular.
In this subsection, we give a necessary and sufficient condition
for Frobenius kernels of a split quasireductive supergroup to be unimodular in terms of the root system of it.

Let $\G$ be a split quasireductive supergroup, and let $r$ be a positive integer.
Set $\g:=\Lie(\G)$.
Since the $r$-th Frobenius kernel $\G_r$ of $\G$ is finite and normal,
there uniquely exists $\chi_r\in \gpl(\O(\G)) \cong \X(\G)$ such that
$\coad_{\O(\G_r)}^*(\phi_{\G_r}) = \phi_{\G_r}\otimes \chi_r$
by Lemma~\ref{prp:dist-grp}.
Here, $\phi_{\G_r}$ is a fixed non-zero left integral for $\G_r$.
As a super-analogue of \cite[Part~I, Proposition~9.7]{Jan03},
Zubkov and Marko \cite{ZubMar16} explicitly determined the value of $\chi_r$ as follows.
\begin{prop}[{\cite[Proposition~6.11]{ZubMar16}}] \label{prp:ZubMar}
Let $R$ be a commutative superalgebra.
For each $g\in \G(R)$,
\[
\chi_r(g) = \Ber(\Ad(g))^{p^r-1}\cdot \det_{\odd}(\Ad(g))^{p^r}.
\]
Here, the left adjoint action $\Ad(g)$ on $\g$ is regarded as an element of\linebreak[4] $\Mat_{\dim(\g_\eve)|\dim(\g_\odd)}(R)$ with respect to a fixed homogeneous basis of $\g$.
\end{prop}

Set $T_r:=T\cap (\Gev)_r$.
Note that, $T_r$ is the $r$-th Frobenius kernel of $T$.
The following is a version of Lemma~\ref{prp:character2}:
\begin{lemm}\label{prp:character3}
The map $\X(\G_r)\to \X(T_r);\; \chi\mapsto \chi|_{T_r}$ is injective.
\end{lemm}
\begin{proof}
For each $\alpha\in\Deltaa_\eve$, let $(U_\alpha)_r$ denote the $r$-th Frobenius kernel of the $\alpha$-root subgroup $U_\alpha$ of $\Gev$.
Since $U_\alpha \cong \Ga$, one sees that the corresponding Hopf algebra of $(U_\alpha)_r$ is
isomorphic to the quotient $\Bbbk[X_\alpha]/(X_\alpha^{p^r})$ of the polynomial algebra $\Bbbk[X_\alpha]$.
Thus, the character group of $(U_\alpha)_r$ is trivial.
Since $(\Gev)_r$ is generated by $(U_\alpha)_{r}$ and $T_r$,
the map $\X((\Gev)_r)\to \X(T_r)$; $\chi\mapsto\chi|_{T_r}$ is injective.
Then by Lemma~\ref{prp:character} and Proposition~\ref{prp:Gevr=Grev}, we are done.
\end{proof}

Recall that $\X(T)\cong \mathz^\ell =\bigoplus_{i=1}^\ell \mathz\lambda_i$.
We shall write down the odd roots by the basis.
For each $\gamma\in\Deltaa_\odd$, there uniquely exits $n(\gamma)_1,\dots,n(\gamma)_\ell\in\mathz$ such that
\[
\gamma = \sum_{i=1}^\ell n(\gamma)_i \lambda_i.
\]
Using this notation, we have the following result:
\begin{prop} \label{lem:Frob-unimo}
The $r$-th Frobenius kernel $\G_r$ of $\G$ is unimodular
if and only if $\sum_{\gamma\in\Deltaa_{\odd}}\dim(\g_\odd^\gamma) n(\gamma)_i \in p^r\mathz$ for all $1\leq i\leq \ell$.
\end{prop}
\begin{proof}
By Theorem~\ref{prp:criterion-unimo},
we have $\G_r$ is unimodular if and only if the restriction $\chi_r|_{\G_r}$ is trivial.
On the other hand, by Lemma~\ref{prp:character3}, the restriction $\chi_r|_{\G_r}$ is trivial if and only if $\chi_r|_{T_r}$ is trivial.

Let $R$ be a commutative algebra.
By the explicit description of $\chi_r$ (Proposition~\ref{prp:ZubMar}),
for each $t\in T_r(R)$
\begin{eqnarray*}
\chi_r(t)
&=& \det_\eve (\Ad(t))^{p^r-1}\cdot \det_\odd (\Ad(t)) \\
&=& \prod_{\alpha\in\Deltaa_\eve}\alpha(t)^{p^r-1} \cdot \prod_{\gamma\in\Deltaa_\odd} \gamma(t)^{\dim(\g_\odd^\gamma)}
\quad =\quad \prod_{\gamma\in\Deltaa_\odd} \gamma(t)^{\dim(\g_\odd^\gamma)}.
\end{eqnarray*}
Here, the last equation follows from $\sum_{\alpha\in\Deltaa_\eve}\alpha=\mathbf{0}$ in $\X(T)$.

Recall that, the identification $\X(T)\cong \mathz^\ell$ is induced from the fixed isomorphism $T\cong \Gm^\ell$.
Since $T_r\cong \bmu_{p^r}^\ell$, we have $\X(T_r)\cong (\mathz/p^r\mathz)^\ell$ through this identification.
For each $1\leq i\leq \ell$, we get
\[
t=(1,\dots,\overset{i}{\check{t_i}},\dots,1)\in T_r(R)\cong \bmu_{p^r}^\ell(R)
\,\,\Longrightarrow\,\,
\chi_r(t) = \prod_{\gamma\in\Deltaa_\odd} t_i^{\dim(\g_\odd^\gamma) n(\gamma)_i}.
\]
Thus, $\chi_r|_{T_r}$ is trivial if and only if $\sum_{\gamma\in\Deltaa_{\odd}}\dim(\g_\odd^\gamma) n(\gamma)_i \in\mathz$
is divided by $p^r$ for each $1\leq i\leq \ell$.
This proves the claim.
\end{proof}

\begin{thm} \label{prp:Frob-unimo}
The following conditions are equivalent:
\begin{enumerate}
\item\label{prp:Frob-unimo(1)}
For all positive integer $r$, the $r$-th Frobenius kernel $\G_r$ of $\G$ is unimodular.
\item\label{prp:Frob-unimo(2)}
$\sum_{\gamma\in\Deltaa_{\odd}}\dim(\g_\odd^\gamma) \gamma=\mathbf{0}$ in $\X(T)$.
\end{enumerate}
\end{thm}
\begin{proof}
By Proposition~\ref{lem:Frob-unimo},
it follows that the condition \eqref{prp:Frob-unimo(1)} holds if and only if
$\sum_{\gamma\in\Deltaa_{\odd}}\dim(\g_\odd^\gamma) n(\gamma)_i =0$ for all $1\leq i\leq \ell$.
The last condition is obviously equivalent to \eqref{prp:Frob-unimo(2)}.
The proof is done.
\end{proof}

\begin{cor}\label{cor:Frob-unimo}
Let $\G$ be one of $\GL(m|n)$, $\Q(n)$ or a Chevalley supergroup of classical type.
For any positive integer $r$, the $r$-th Frobenius kernel $\G_r$ of $\G$ is unimodular.
\end{cor}
\begin{proof}
The argument in the proof of Corollary~\ref{cor:quasi-unimo} gives $\sum_{\gamma\in\Deltaa_\odd}\dim(\g_\odd^\gamma)\gamma=\mathbf{0}$ in $\X(T)$ also in positive characteristic.
Thus, by Theorem~\ref{prp:Frob-unimo}, we are done.
\end{proof}

\begin{ex} \label{ex:non-unimo_p}
As we have seen in Example~\ref{ex:non-unimo_0},
the $r$-th Frobenius kernels of
the periplectic supergroup $\P(n)$
and the supergroup $\G$ defined in Example~\ref{ex:non-unimo_0}\eqref{ex:non-unimo_0(2)} are non-unimodular.
For the supergroup $\F^{\langle g_1,\dots,g_n\rangle}=F\ltimes (\Gao)^n$ given in Example~\ref{ex:root-sys}\eqref{ex:root-sys(4)},
the $r$-th Frobenius kernel $(\F^{\langle g_1,\dots,g_n\rangle})_r$ is unimodular if and only if
$\sum_{j=1}^m d_j \chi_{i_j}\in p^r\X(T)$,
by Proposition~\ref{lem:Frob-unimo} and Example~\ref{ex:non-unimo_0}\eqref{ex:non-unimo_0(3)}.
\qedd
\end{ex}

\section{Steinberg's Tensor Product Theorem} \label{sec:Ste}
We fix a split quasireductive supergroup $\G$ with a split maximal torus $T$ of $\Gev$.
Set $\g:=\Lie(\G)$ and $\h:=\g^{\mathbf{0}}$ as before.
In this section,
we establish Steinberg's tensor product theorem for $\G$ under natural assumptions.

\subsection{Simple $\G$-supermodules} \label{sec:simple-G-supermods}
In \cite{Shi20}, we defined some special closed super-subgroups of $\G$ and constructed all simple left $\G$-supermodules.
In the following, we briefly review the construction.

First of all, we constructed a closed super-subgroup $\T$ of $\G$ such that $\Tev=T$ with $\Lie(\T)=\h$.
We fix a group homomorphism $\Upsilon:\mathz\Deltaa\to\mathbb{R}$
with $\Upsilon(\Deltaa\setminus\{\mathbf{0}\})\subset\mathbb{R}\setminus\{0\}$ to define an ``order'' on $\Deltaa$ as follows:
\[
\Deltaa^\pm:=\{\alpha\in\Deltaa\setminus\{\mathbf{0}\} \mid \pm \Upsilon(\alpha) >0\},
\quad
\Deltaa_\epsilon^\pm:=\Deltaa_\epsilon \cap \Deltaa^\pm
\quad
(\epsilon\in\mathz_2).
\]
Along this order, we can construct a closed super-subgroup $\B^+$ (resp.~$\B$) of $\G$, called the \defnote{Borel super-subgroup} of $\G$,
such that $\Bevp$ (resp.~$\Bev$) is a positive (resp.~negative) Borel subgroup of $\Gev$
with respect to $\Deltaa_\eve^+$ (resp.~$\Deltaa_\eve^-$) satisfying
\[
\Lie(\B^+) = \h\oplus \bigoplus_{\alpha\in\Deltaa^+} \g^\alpha
\quad
(\text{resp.}~\Lie(\B) = \h\oplus \bigoplus_{\alpha\in\Deltaa^-} \g^\alpha).
\]
Also, we can construct a closed super-subgroup $\U^+$ of $\G$ such that $\Uevp$ is a unipotent subgroup of $\Gev$
and $\Lie(\U^+)=\bigoplus_{\alpha\in\Deltaa^+}\g^\alpha$.
One sees that $\U^+$ is {\it unipotent},
that is,
for any non-zero left $\U^+$-supermodule,
its $\U^+$-invariant super-subspace is non-zero, see \cite{ZubUly14}.
(In other words,
the corresponding Hopf superalgebra $\O(\U^+)$ is irreducible, see also \cite[Definition~4(1)]{Mas12}).
Moreover, we have $\B^+\cong \T\ltimes\U^+$.

Analogously, we can find $\U\subset \G$ such that $\B\cong \T\ltimes \U$.
\medskip

Using Clifford superalgebra theory,
we can find a simple left $\T$-supermodule $\uu(\lambda)$ for each $\lambda\in \X(T)$.
Moreover, the map $\X(T)\to \Simple_\Pi(\T)$; $\lambda \mapsto \uu(\lambda)$ is bijective.
As left $T$-modules, this $\uu(\lambda)$ is isomorphic to a (finite) copy of the one-dimensional left $T$-module $\Bbbk^\lambda$.
Thus, if $\T=T$ (i.e., $\mathbf{0}\notin\Deltaa$), then $\uu(\lambda)$ is just $\Bbbk^\lambda$.
\begin{lemm}\label{prp:character-triv}
The restriction map $\X(\T)\to\X(T);\ \chi\mapsto\chi|_T$ is injective, and its image is $\{\lambda\in\X(T)\mid\dim(\uu(\lambda))=1\}$.
\end{lemm}
\begin{proof}
Injectivity follows from Lemma~\ref{prp:character}.
By the classification of simple left $\T$-supermodules, a weight $\lambda\in\X(T)$ is the restriction of a character of $\T$ if and only if $\dim(\uu(\lambda))=1$.
\end{proof}

Since $\T$ is a closed super-subgroup of $\B$,
we may regard $\uu(\lambda)$ as a left $\B$-supermodule (i.e., a right $\O(\B)$-supercomodule),
which we denote by the same symbol.
For each $\lambda\in\X(T)$, we get a left $\G$-supermodule
\[
H^0(\lambda) := \ind_{\B}^{\G}(\uu(\lambda)) = \uu(\lambda) \cotens_{\O(\B)} \O(\G),
\]
where $\cotens_{\O(\B)}$ denotes the cotensor product over $\O(\B)$
and $\O(\G)$ is naturally regarded as a left $\O(\B)$-supercomodule via $\B\subset \G$.

Set
\[
\X(T)^\flat:=\{\lambda\in\X(T) \mid H^0(\lambda)\neq0\}.
\]
For each $\lambda\in\X(T)^\flat$,
we can show that $H^0(\lambda)$ has a unique simple left $\G$-super-submodule $L(\lambda)$.
\begin{thm}[{\cite[Theorem~4.12 and Proposition~4.15]{Shi20}}] \label{prp:simple-G}
The map $\X(T)^\flat\to \Simple_\Pi(\G)$; $\lambda \mapsto L(\lambda)$ is bijective.
Moreover, $\lambda$ is a ``highest'' $T$-weight of $L(\lambda)$,
in the sense that
the $\lambda$-weight superspace $L(\lambda)^\lambda$ is isomorphic to $\uu(\lambda)$ as left $\T$-supermodules
and the action of $\U^+$ on $L(\lambda)^\lambda$ is trivial.
\end{thm}
\begin{proof}
The classification is \cite[Theorem~4.12]{Shi20}.
For the highest weight assertion, \cite[Lemma~4.14]{Shi20} gives $H^0(\lambda)^{\U^+}\cong\uu(\lambda)$ as left $\T$-supermodules.
Every maximal $T$-weight $\nu$ of the finite-dimensional supermodule $L(\lambda)$ satisfies $L(\lambda)^\nu\subset H^0(\lambda)^{\U^+}$, so $\nu=\lambda$.
Thus all weights of $L(\lambda)$ are less than or equal to $\lambda$, and $L(\lambda)^\lambda$ is a non-zero left $\T$-super-submodule of $H^0(\lambda)^{\U^+}$.
The simplicity of $\uu(\lambda)$ proves the claim.
\end{proof}

Note that in the proof of \cite[Proposition~4.15]{Shi20}, the injective map should be the evaluation map $H^0(\lambda)^{\U^+}\to\uu(\lambda)$ induced by $\id\otimes\varepsilon_{\G}$.
The statement of the proposition is unchanged.

\begin{deff}\label{def:abs-irr}
A simple left $\G$-supermodule $L$ is said to be \defnote{absolutely simple}
if $L\otimes \Bbbk'$ is a simple left $\G_{\Bbbk'}$-supermodule for all field extensions $\Bbbk'$ of $\Bbbk$.
Here, $\G_{\Bbbk'}$ denotes the base change of $\G$ to $\Bbbk'$.
\end{deff}

The following is a corollary of Theorem~\ref{prp:simple-G}:
\begin{cor}
Let $\lambda\in\X(T)^\flat$.
If $L(\lambda)$ is absolutely simple,
then $L(\lambda)\otimes \Bbbk'\cong L_{\Bbbk'}(\lambda)$ as left $\G_{\Bbbk'}$-supermodules for all field extensions $\Bbbk'$ of $\Bbbk$.
Here, $L_{\Bbbk'}(\lambda)$ denotes a simple left $\G_{\Bbbk'}$-supermodule of highest weight $\lambda$.
\end{cor}

\begin{prop} \label{prp:parameter}
For a field extension $\Bbbk'$ of $\Bbbk$,
we have $\X(T_{\Bbbk'})^\flat \subset \X(T)^\flat$.
If $\mathbf{0}\notin\Deltaa$ (or equivalently, $\T=T$),
then $\X(T_{\Bbbk'})^\flat = \X(T)^\flat$.
\end{prop}
\begin{proof}
After replacing $\uu_{\Bbbk'}(\lambda)$ by its parity change if necessary, it embeds into $\uu(\lambda)\otimes\Bbbk'$.
Since the induction functor $\ind_{\B}^{\G}(-)$ commutes with all field extensions,
\[
H^0(\lambda)\otimes \Bbbk'
\cong \big(\uu(\lambda)\otimes \Bbbk'\big) \cotens_{\O(\B_{\Bbbk'})} \O(\G_{\Bbbk'})
\supset
\uu_{\Bbbk'}(\lambda) \cotens_{\O(\B_{\Bbbk'})} \O(\G_{\Bbbk'})
=:H^0_{\Bbbk'}(\lambda).
\]
Here, $\uu_{\Bbbk'}(\lambda)$ denotes a simple left $\T_{\Bbbk'}$-supermodule with weight $\lambda$, unique up to parity change.
Thus, $H^0_{\Bbbk'}(\lambda)\neq0$ implies $H^0(\lambda)\neq0$.
If $\mathbf{0}\notin\Deltaa$,
then $\uu(\lambda) = \Bbbk^\lambda$,
and hence $H^0_{\Bbbk'}(\lambda)\cong H^0(\lambda)\otimes\Bbbk'$.
\end{proof}

\begin{prop} \label{prp:T=Tev-abs-simple}
If $\mathbf{0}\notin\Deltaa$ (or equivalently, $\T=T$),
then all simple left $\G$-supermodules are absolutely simple.
In particular, for a field extension $\Bbbk'$ of $\Bbbk$ and a left $\G$-supermodule $V$,
we have (1) $\soc_{\G}(V)\otimes \Bbbk' \cong\soc_{\G_{\Bbbk'}}(V\otimes \Bbbk')$;
and (2) $V$ is $\G$-semisimple if and only if $V\otimes\Bbbk'$ is $\G_{\Bbbk'}$-semisimple.
\end{prop}
\begin{proof}
Let $\lambda\in\X(T)^\flat$.
We note that $L(\lambda)^\lambda\cong \uu(\lambda) = \Bbbk^\lambda$ and $L(\lambda)\not\cong\Pi L(\lambda)$.
By Frobenius reciprocity (see \cite[Section~A.3]{Shi20}), we get
\[
{}_{\G}\underline{\Hom}(L(\lambda),H^0(\lambda))
\cong {}_{\B}\underline{\Hom}(L(\lambda),\Bbbk^\lambda)
\subset {}_{T}\underline{\Hom}(L(\lambda)^\lambda,\Bbbk^\lambda)
\cong \Bbbk.
\]
Since $\id_{L(\lambda)} \in {}_{\G}\underline{\End}(L(\lambda))$,
we can conclude that
${}_{\G}\underline{\End}(L(\lambda))=\Bbbk$,
and hence ${}_{\hy(\G)}\underline{\End}(L(\lambda))=\Bbbk$ by Theorem~\ref{thm:G=loc-fin-hy(G)-T}.
Let $\rho:\hy(\G)\to \underline{\End}_\Bbbk(L(\lambda))$ denote the $\hy(\G)$-supermodule structure map of $L(\lambda)$.
Then by {\it Jacobson density theorem} for superalgebras \cite{Rac98},
the above argument implies that $\rho$ is surjective.

Let $\Bbbk'$ be a filed extension of $\Bbbk$.
Since $\G_\mathz$ is infinitesimally flat and $\rho$ is surjective,
the $\hy(\G_{\Bbbk'})$-supermodule structure map
\[
\rho\otimes\Bbbk' : \hy(\G_{\Bbbk'}) \cong \hy(\G)\otimes \Bbbk' \longrightarrow
\underline{\End}_{\Bbbk'}(L(\lambda)\otimes \Bbbk') \cong \underline{\End}_\Bbbk(L(\lambda)) \otimes \Bbbk',
\]
of $L(\lambda)\otimes \Bbbk'$ is also surjective.
In general,
it is easy to see that $L(\lambda)\otimes \Bbbk'$ is a simple left $\underline{\End}_{\Bbbk'}(L(\lambda)\otimes\Bbbk')$-supermodule.
Therefore, we conclude that $L(\lambda)\otimes \Bbbk'$ is a simple left $\hy(\G_{\Bbbk'})$-supermodule.
By Theorem~\ref{thm:G=loc-fin-hy(G)-T} (for $\G_{\Bbbk'}$),
$L(\lambda)\otimes \Bbbk'$ is a simple left $\G_{\Bbbk'}$-supermodule.

By counting multiplicity of simple super-submodules inside of $V$, the claim (1) easily follows.
The claim (2) is just a consequence of (1).
\end{proof}

In the non super-situation,
it is known that all left simple $\Gev$-modules are absolutely simple,
see \cite[Part~II, Corollary~2.9]{Jan03} (and \cite[Section~22.4]{Mil17}).
However, the following example shows that this phenomenon is no longer true for the super-situation
when $\mathbf{0}\in\Deltaa$ (or equivalently, $\T\neq T$):
\begin{ex} \label{ex:non-abs-irr}
Suppose that our base field $\Bbbk$ satisfies $-1\notin(\Bbbk^\times)^2$,
that is, $\Bbbk$ does not contain $x$ such that $x^2=-1$.
Let $\G$ be the queer supergroup $\Q(2)$ over $\Bbbk$.
Take $T$ to be the standard maximal torus of $\Gev\cong \GLL_2$
and identify $\X(T)$ with $\mathz\lambda_1\oplus \mathz\lambda_2$ as before.
By construction (\cite[Section~4.1]{Shi20}), the simple left $\T$-supermodule $\uu(\lambda)$ is a simple supermodule, unique up to isomorphism and parity change, over the Clifford superalgebra $\Cl(\h_\odd,b^\lambda)$ of $\h_\odd=\Lie(\T)_{\odd}$ with the symmetric bilinear form
\[
b^\lambda: \h_\odd\times \h_\odd \longrightarrow \Bbbk;
\quad
(x,y)\longmapsto \lambda([x,y]).
\]
Let $\lambda=-\lambda_1-4\lambda_2\in\X(T)$.
For $\Q(n)$, by \cite[Theorem~6.11]{BruKle03}, we know that
\begin{equation} \label{eq:Q(n)-dominant}
\X(T_{\overline{\Bbbk}})^\flat = \{
\sum_{i=1}^n c_i \lambda_i \in\bigoplus_{i=1}^n \mathz\lambda_i \mid c_1\geq \cdots \geq c_n
\text{ and } [c_i=c_{i+1} \Rightarrow p\mid c_i] \},
\end{equation}
where $\overline{\Bbbk}$ denotes the algebraic closure of $\Bbbk$ and $p:=\mathsf{char}(\Bbbk)$.
Thus, by Proposition~\ref{prp:parameter},
we have $\lambda\in \X(T_{\overline{\Bbbk}})^\flat \subset \X(T)^\flat$.
Let $K_1,K_2$ denote the images of the standard odd diagonal matrices in $\Cl(\h_\odd,b^\lambda)$.
They satisfy $K_1^2=-1$, $K_2^2=-4$ and $K_1K_2=-K_2K_1$.
Setting $\mathtt{i}:=K_1$ and $\mathtt{j}:=K_2/2$, we see that $\Cl(\h_\odd,b^\lambda)$ is isomorphic to the quaternion superalgebra
\[
\left(\frac{-1,-1}{\Bbbk}\right):=\Bbbk\langle\mathtt{i},\mathtt{j}\rangle/(\mathtt{i}^2+1,\mathtt{j}^2+1,\mathtt{i}\mathtt{j}+\mathtt{j}\mathtt{i}),
\]
where $|\mathtt{i}|=|\mathtt{j}|=1$.
Its even part $\Bbbk[\mathtt{i}\mathtt{j}]\cong\Bbbk[X]/(X^2+1)$ is a field, and every non-zero homogeneous element is invertible.
Thus its regular supermodule is simple, and $\uu(\lambda)$ is a $4$-dimensional vector space over $\Bbbk$.

On the other hand,
since the base change of $\Cl(\h_\odd,b^\lambda)$ to the filed $\Bbbk':=\Bbbk[X]/(X^2+1)$
is isomorphic to the matrix superalgebra $\Mat_{1|1}(\Bbbk')$,
its simple supermodule is a $2$-dimensional vector space over $\Bbbk'$,
which we denote by $\uu_{\Bbbk'}(\lambda)$.
The maximal $T_{\Bbbk'}$-weight of $L(\lambda)\otimes\Bbbk'$ is still $\lambda$, and
\[
\dim\big((L(\lambda)\otimes\Bbbk')^\lambda\big)=\dim\big(\uu(\lambda)\otimes\Bbbk'\big)=4>2=\dim(\uu_{\Bbbk'}(\lambda)).
\]
If $L(\lambda)\otimes\Bbbk'$ were simple, Theorem~\ref{prp:simple-G} would identify its $\lambda$-weight superspace with $\uu_{\Bbbk'}(\lambda)$ up to parity change.
This is a contradiction, so $L(\lambda)$ is not absolutely simple.
\qedd
\end{ex}

\subsection{Simple $\G_r$-supermodules} \label{sec:simple-G_r-supermods}
Throughout the rest of the paper,
we suppose that $\Bbbk$ is a perfect field of characteristic $p>2$.
The Clifford superalgebra defining $\uu(\lambda)$ depends only on the class of $\lambda$ in $\X(T)/p\X(T)$.
For each class, we fix a simple supermodule over the corresponding Clifford superalgebra and use this same choice to define $\uu(\lambda)$ for every weight in the class, with $T$ acting by $\lambda$.
For the zero class, we choose the purely even one-dimensional module, so $\uu(0)=\Bbbk$.

In the following, we fix a positive integer $r$.
As in Section~\ref{sec:Frob-ker}, the $r$-th Frobenius kernel $\B_r^+$ (resp.~$\B_r$) of $\B^+$ (resp.~$\B$) are infinitesimal and normal.

Recall that, for each $\lambda\in\X(T)$, we have regarded the simple left $\T$-supermodule $\uu(\lambda)$ as a left $\B$-supermodule.
We also regard $\uu(\lambda)$ as a left $\B_r$-supermodule (resp.~$\T_r$-supermodule)
via the inclusion $\B_r \subset \B$ (resp.~$\T_r \subset \B$),
which we again denote by the same symbol.
By Proposition~\ref{prp:Gevr=Grev}, we have $(\T_r)_\ev = T_r$, and hence we get the following result:
\begin{lemm} \label{prp:u(lambda+p^mu)}
For $\lambda,\mu\in\X(T)$,
we have $\uu(\lambda+p^r\mu) \cong \uu(\lambda)$ as left $\B_r$-supermodules.
\end{lemm}

By definition, we get the short exact sequence
$\mathbf{0}\to p^r\X(T)
\hookrightarrow \X(T)
\twoheadrightarrow \X(T_r)
\to\mathbf{0}$,
where $\X(T)\to\X(T_r)$ is the restriction map induced from $T_r\subset T$.
Thus, by Lemma~\ref{prp:u(lambda+p^mu)}, for each $\lambda\in\X(T_r)$,
we can define a left $\B_r$-supermodule structure on $\uu(\lambda)$ in an obvious way.

\begin{prop} \label{prp:L_r-simple}
For each $\lambda\in\X(T_r)$,
the induced left $\G_r$-supermodule $\ind_{\B_r}^{\G_r}(\uu(\lambda))$ of $\uu(\lambda)$ has a unique simple left $\G_r$-super-submodule $L_r(\lambda)$.
Moreover, the map $\X(T_r) \to \Simple_\Pi(\G_r);\; \lambda \mapsto L_r(\lambda)$ is bijective.
\end{prop}
\begin{proof}
Since the superalgebra map $\O(\G_r)\to \O(\B_r^+)\otimes\O(\B_r)$ induced from the multiplication on $\G_r$ is injective,
one easily sees that the same argument as in \cite[Theorem~4.12]{Shi20} works for the quadruple $(\G_r,\B_r^+,\B_r,\T_r)$.
Thus, to prove the claim, it is enough to show that $\X(T_r)=\{\lambda\in\X(T_r) \mid \ind_{\B_r}^{\G_r}(\uu(\lambda))\neq 0 \}$.

It is easy to see that the multiplication on $\G_r$ induces an isomorphism $\B_r\times \U_r^+ \to \G_r$ of superschemes,
where $\U_r^+$ is the $r$-th Frobenius kernel of $\U^+$.
Since the isomorphism is compatible with the left $\B_r$-multiplication,
we get an isomorphism $\O(\G_r)\cong \O(\B_r)\otimes \O(\U_r^+)$ of left $\O(\B_r)$-supercomodules.
Thus, we have
\[
\ind_{\B_r}^{\G_r}(\uu(\lambda)) = \uu(\lambda)\cotens_{\O(\B_r)}\O(\G_r)
\cong \uu(\lambda)\otimes \O(\U_r^+) \neq 0
\]
for each $\lambda\in\X(T_r)$.
Thus, we are done.
\end{proof}

The proof of Proposition~\ref{prp:L_r-simple} shows that
the dimension of $\ind^{\G_r}_{\B_r}(N)$ is given by $p^{r\#\Deltaa_\eve^+}\cdot 2^{n_\odd^+}\cdot \dim(N)$ for each left $\B_r$-supermodule $N$,
where $n_\odd^+:=\sum_{\gamma\in\Deltaa_\odd^+}\dim(\g_\odd^\gamma)$.
Since $\X(T_r) \cong \X(T)/p^r\X(T)$, we get the following result:
\begin{prop} \label{prp:L_r(lambda+p^rmu)}
For all $\lambda,\mu\in\X(T)$,
we have $L_r(\lambda+p^r\mu) \cong L_r(\lambda)$ as left $\G_r$-supermodules.
\end{prop}

We get the following result,
whose proof is similar to that of Proposition~\ref{prp:T=Tev-abs-simple}.
\begin{prop} \label{prp:T=Tev-abs-simple_r}
If $\mathbf{0}\notin\Deltaa$ (or equivalently, $\T=T$),
then all simple left $\G_r$-supermodules are absolutely simple.
In particular, for a field extension $\Bbbk'$ of $\Bbbk$ and a left $\G_r$-supermodule $V$,
we have (1) $\soc_{\G_r}(V)\otimes \Bbbk' \cong\soc_{(\G_r)_{\Bbbk'}}(V\otimes \Bbbk')$;
and (2) $V$ is $\G_r$-semisimple if and only if $V\otimes\Bbbk'$ is $(\G_r)_{\Bbbk'}$-semisimple.
\end{prop}

Recall that, in Section~\ref{sec:unimod-Frob-ker},
we have used the character $\chi_r\in\X(\G)$ to discuss the unimodularity of $\G_r$.
Since $\B_r^+$ is infinitesimal and normal,
by Lemma~\ref{prp:dist-grp}, we also find a unique character $\psi_r\in \gpl(\O(\B^+)) \cong \X(\B^+)$ such that
$\coad_{\O(\B_r^+)}^*(\phi_{\B_r^+}) = \phi_{\B_r^+} \otimes \psi_r$,
where $\coad_{\O(\B_r^+)}^*:\O(\B_r^+)^*\to \O(\B_r^+)^*\otimes \O(\B^+)$ is the induced right $\O(\B^+)$-coaction on $\O(\B_r^+)^*$
and $\phi_{\B_r^+}$ is a fixed non-zero left integral for $\B_r^+$.
In the following, we let $\epsilon_r$ denote the sum of the parity of the integral $\phi_{\G_r}$ and $\phi_{\B_r^+}$,
see Proposition~\ref{prp:finite-integral}.

We set
\[
\delta_r := \chi_r|_{\B^+}\cdot \psi_r^{-1}.
\]
as an element of $\X(\B^+)\cong\gpl(\O(\B^+))$.
Since $-\Deltaa_\eve^-=\Deltaa_\eve^+$ and $\Deltaa_\odd=\Deltaa_\odd^+\sqcup\Deltaa_\odd^-$, we have
\begin{equation} \label{eq:delta_r}
\delta_r|_T = -(p^r-1)\sum_{\alpha\in\Deltaa_\eve^+}\alpha + \sum_{\gamma\in\Deltaa_\odd^-}\dim(\g_\odd^\gamma)\gamma.
\end{equation}
Here, we write the group low of $\X(T)$ additively.
In particular,
$\delta_r|_{T_r}=\sum_{\alpha\in\Deltaa_\eve^+}\alpha+\sum_{\gamma\in\Deltaa_\odd^-}\dim(\g_\odd^\gamma)\gamma$.
In the following, we simply write $\delta_r|_{\B^+}$ and $\delta_r|_{\B_r^+}$ by $\delta_r$.
\medskip

For a left $\B_r^+$-supermodule $N$, we set
\[
\coind_{\B_r^+}^{\G_r}(N) := \hy(\G_r)\otimes_{\hy(\B_r^+)} N.
\]
Note that, the dimension of $\coind_{\B_r^+}^{\G_r}(N)$ is $p^{r\#\Deltaa_\eve^-}\cdot 2^{n_\odd^-} \dim(N)$
by the tensor decomposition $\hy(\G_r)\cong \hy(\U_r)\otimes \hy(\B_r^+)$, see Theorem~\ref{prp:PBW-Gr}.
Here, we put $n_\odd^-:=\sum_{\gamma\in\Deltaa_\odd^-}\dim(\g_\odd^\gamma)$.
In particular, if $N$ is finite, then so is $\coind_{\B_r^+}^{\G_r}(N)$.

Since $\G_r$ is infinitesimal, we have $\hy(\G_r)=\O(\G_r)^*$,
and hence we may naturally regard $\coind_{\B_r^+}^{\G_r}(N)$ as a left $\G_r$-supermodule.
Marko and Zubkov showed the following result:
\begin{prop}[{\cite[Proposition~13 and Lemma~14]{MarZub18}}] \label{prp:MZ_Prop13}
Let $N$ be a left $\B_r^+$-supermodule.
Then there is an isomorphism
$\coind_{\B_r^+}^{\G_r}(N) \cong \Pi^{\epsilon_r} \ind_{\B_r^+}^{\G_r}(N \otimes \delta_r)$
of left $\G_r$-supermodules.
If $N$ is finite, then $\coind_{\B_r^+}^{\G_r}(N)^* \cong \ind_{\B_r^+}^{\G_r}(N^*)$.
\end{prop}
\begin{proof}
Put $A=\hy(\G_r)$ and $B=\hy(\B_r^+)$.
All Hom spaces below are internal super Hom spaces.
For $C\in\{A,B\}$, choose a nonzero left integral $t_C$ and put $\epsilon_C=|t_C|\in\mathbb Z_2$.
Let $\theta_C$ be the character of the adjoint action on $\Bbbk t_C$.
Choose a nonzero cointegral $\ell_C$ on $C$, normalized by $\ell_C(t_C)=1$.
Then $\ell_C$ has parity $\epsilon_C$, and the bilinear form $(x,y)\mapsto\ell_C(xy)$ is nondegenerate.
A linear functional with this property is sometimes called a Frobenius functional.
Conjugation equivariance gives
$\ell_C(\ad(x)y)=\theta_C(x)\ell_C(y)$.
Using super-cocommutativity and
$xy=\sum_x(-1)^{|x_{(2)}||y|}\ad(x_{(1)})(yx_{(2)})$,
we obtain
\[
\ell_C(xy)=(-1)^{|x||y|}\ell_C(y\nu_C(x)),
\qquad \nu_C(x)=\sum_x\theta_C(x_{(1)})x_{(2)}.
\]
Thus $C^*\cong\Pi^{\epsilon_C}({}_C C_{\nu_C})$ as superbimodules,
where the subscript twists the right action.
Since $A$ is finite free as a left $B$-supermodule, we obtain
\[
\underline{\Hom}_B(A,B)
\cong \Pi^{\epsilon_r}{}_A A_{\tau_{\delta_r}},
\qquad
\tau_{\delta_r}(b)=\sum_b\delta_r(b_{(1)})b_{(2)}.
\]
Here $\epsilon_r=\epsilon_A+\epsilon_B$ in $\mathbb Z_2$.
Indeed, this follows from
$\underline{\Hom}_B(A,B^*)\cong A^*$ and
$\nu_A|_B\circ\nu_B^{-1}=\tau_{\delta_r}$.
Consequently,
\[
\ind_{\B_r^+}^{\G_r}(N)
\cong \underline{\Hom}_B(A,B)\otimes_B N
\cong \Pi^{\epsilon_r}\coind_{\B_r^+}^{\G_r}(N\otimes\delta_r^{-1}).
\]
Replacing $N$ by $N\otimes\delta_r$ proves the first assertion.
For finite $N$, the second assertion follows from tensor--Hom duality.
\end{proof}

For each $\lambda\in\X(T)$,
we set
\[
M_r(\lambda) := \coind_{\B_r^+}^{\G_r}(\uu(\lambda)).
\]
Note that, $M_r(\lambda)$ is finite-dimensional.
Let $\rad_{\G_r}(M_r(\lambda))$ denote the $\G_r$-radial of $M_r(\lambda)$,
that is, the intersection of all maximal $\G_r$-super-submodules of $M_r(\lambda)$.
Set $\topp_{\G_r}(M_r(\lambda)):=M_r(\lambda)/\rad_{\G_r}(M_r(\lambda))$.
\begin{prop} \label{prp:L_r=M_r/Rad}
For each $\lambda\in\X(T)$, $\topp_{\G_r}(M_r(\lambda))$ is isomorphic to $L_r(\lambda)$ as left $\G_r$-supermodules.
\end{prop}
\begin{proof}
First, we show that $\topp_{\G_r}(M_r(\lambda))$ is simple.
In the argument proving simplicity of the top, the identifications involving $\uu(\lambda)$ are understood up to parity change.
To show this, we note that for any $\lambda\in\X(T)$, $M_r(\lambda)$ has a unique simple $\G_r$-super-submodule.
Indeed, by Proposition~\ref{prp:MZ_Prop13}, we get
\[
M_r(\lambda)
\cong \Pi^{\epsilon_r}\ind_{\B_r^+}^{\G_r}\big(\uu(\lambda)\otimes \delta_r\big)
\cong \Pi^{\epsilon_r}\ind_{\B_r^+}^{\G_r}\big(\uu(\lambda+\delta_r)\big).
\]
Then by mimicking the proof given in Proposition~\ref{prp:L_r-simple}, this proves the claim.
The dual of $\topp_{\G_r}(M_r(\lambda))$ can be naturally regarded as a $\G_r$-super-submodule of $M_r(\lambda)^*$.
Since $M_r(\lambda)$ is finite-dimensional, $\topp_{\G_r}(M_r(\lambda))$ is semisimple as a left $\G_r$-supermodule.
Thus, to prove that $\topp_{\G_r}(M_r(\lambda))$ is simple,
it is enough to show that $M_r(\lambda)^*$ has a unique simple $\G_r$-super-submodule.
By Proposition~\ref{prp:MZ_Prop13} and $\uu(\lambda)^*=\uu(-\lambda)$, we have
\[
M_r(\lambda)^*
\cong \ind_{\B_r^+}^{\G_r}(\uu(-\lambda))
\cong \Pi^{\epsilon_r} \coind_{\B_r^+}^{\G_r}\big(\uu(-\lambda-\delta_r) \big)
= \Pi^{\epsilon_r} M_r(-\lambda-\delta_r).
\]
Thus, the argument above shows that $\topp_{\G_r}(M_r(\lambda))$ is simple.

The big-cell decomposition used in Proposition~\ref{prp:L_r-simple} gives
$\ind_{\B_r}^{\G_r}(\uu(\lambda))^{\U_r^+}\cong\uu(\lambda)$.
Since $\U_r^+$ is unipotent, $L_r(\lambda)^{\U_r^+}$ is non-zero.
It is a $\T_r$-super-submodule of the simple superspace $\uu(\lambda)$, and hence $L_r(\lambda)^{\U_r^+}\cong\uu(\lambda)$.
Thus, we have
\[
{}_{\G_r}\Hom\big(M_r(\lambda),L_r(\lambda)\big) \cong {}_{\B_r^+}\Hom\big(\uu(\lambda),L_r(\lambda)\big) \neq 0.
\]
We fix a non-zero $\G_r$-homomorphism $M_r(\lambda)\to L_r(\lambda)$, which is surjective since $L_r(\lambda)$ is simple.
By the definition of the radical,
this morphism must factor through the quotient $\topp_{\G_r}(M_r(\lambda))=M_r(\lambda)/\rad_{\G_r}(M_r(\lambda))$.
Since the quotient is simple, the induced morphism is an isomorphism.
\end{proof}

\subsection{Bases of odd roots} \label{sec:odd-bases}
Recall that, $\Deltaa$ is the root system of $\G$ with respect to $T$,
and the quadruple $(\X(T),\Deltaa_\eve,\X(T)^\vee,\Deltaa_\eve^\vee)$ is the root datum of the pair $(\Gev,T)$.
Let $\Psi_\eve$ be the base of $\Deltaa_\eve$ with respect to $\Upsilon$,
in other words,
$\Psi_\eve$ is the set of all simple roots in $\Deltaa_\eve^+$ (see \cite[Chapter~21d]{Mil17}).

By \cite[Part~I, Proposition~7.19 and Remark(2)]{Jan03} (see also \cite[Theorem~2.1]{Tak83}),
as an algebra, $\hy(\Gev)$ is generated by $\hy(T)$ and $\hy(U_{\pm\alpha})$ for $\alpha\in\Psi_\eve$,
where $U_\alpha$ is the $\alpha$-root subgroup of $\Gev$.
Note that, $\hy(U_\alpha)=\bigoplus_{n=0}^{\infty}\Bbbk X_\alpha^{(n)}$
and $\hy(\Uevp)$ is generated by $\{X_\alpha^{(n_\alpha)} \mid \alpha\in\Psi_\eve, n_\alpha\in\mathbb{N} \}$ as an algebra.
By ``$\SLL_2$ theory'',
we get the following commutator formula (see \cite[Section~26]{Hum78} for example):
\begin{equation} \label{eq:comm-formula}
X_\alpha^{(m)}X_{-\alpha}^{(n)} = \sum_{i=0}^{\mathsf{min}\{m,n\}}
X_{-\alpha}^{(n-i)} \binom{H_\alpha-m-n+2i}{i} X_\alpha^{(m-i)}
\end{equation}
for all $m,n\in\mathbb{N}\cup\{0\}$ and $\alpha\in\Deltaa_\eve$,
where $H_\alpha:=[X_\alpha,X_{-\alpha}]$.

In general,
the root system of a split quasireductive supergroup is {\it ill-behaved} (see Example~\ref{ex:root-sys}\eqref{ex:root-sys(4)} for example).
For this reason,
we shall deal with a split quasireductive supergroup having a {\it good} ``simple roots'', as follows:
\begin{deff} \label{asu:gen}
A subset $\Psi_\odd$ of $\Deltaa_\odd^+$ is called a \defnote{\adm base} of $\Deltaa$ if
it satisfies the following three conditions:
\begin{enumerate}
\item\label{asu:gen(2)}
For each $\gamma\in\Deltaa_\odd^+$, the odd root space $\g_\odd^\gamma$ is contained in the Lie super-subalgebra of $\g$ generated by
\[
\bigoplus_{\alpha\in\Deltaa_\eve^+}\g_\eve^\alpha
\quad\text{and}\quad
\{Y_{(\beta,j)}\in\g_\odd^\beta\mid \beta\in\Psi_\odd,\ 1\leq j\leq\dim(\g_\odd^\beta)\}.
\]
\item\label{asu:gen(1)}
For all $\alpha\in\Psi_\eve$ and $\gamma\in\Psi_\odd$ with $\alpha\neq \gamma$, we have $\gamma-\alpha\notin\Deltaa$.
\item\label{asu:gen(3)}
If $\Psi_\eve\cap\Psi_\odd\neq\varnothing$,
then $\dim(\g_\odd^{\pm\alpha})=1$ for all $\alpha\in\Psi_\eve\cap\Psi_\odd$.
\end{enumerate}
In this case, we say that the pair $(\Psi_\eve,\Psi_\odd)$ is a \defnote{\adm base} of $\Deltaa$.
To clarify, we shall say that $\Psi_\eve$ is an \defnote{even base} of $\Deltaa$.
\end{deff}

Note that, $\Psi_\eve$ and $\Psi_\odd$ (if it exists) do depend on the choice of $\Upsilon:\mathz\Deltaa\to\mathr$.

\begin{ex} \label{ex:bases}
We use the notations in Example~\ref{ex:root-sys}.
In the following, we shall extend the domain $\mathz\Deltaa$ of $\Upsilon$ to $\X(T)$ just for simplicity.
\begin{enumerate}
\item\label{ex:bases(1)}
For the general linear supergroup $\GL(m|n)$,
it is natural to take $\Upsilon(\lambda_i):=-i$ for each $1\leq i\leq m+n$.
Then an even base of $\Deltaa$ is given as $\Psi_\eve=\{\lambda_i-\lambda_{i+1}\mid 1\leq i\leq m-1 \text{ or } m+1\leq i\leq m+n-1 \}$
and $\Psi_\odd=\{\lambda_m-\lambda_{m+1}\}$ is a \adm odd base of $\Deltaa$.
\item\label{ex:bases(2)}
For the queer supergroup $\Q(n)$,
we define $\Upsilon(\lambda_i):=-i$ for each $1\leq i\leq n$.
Then $\Psi_\eve=\{\lambda_i-\lambda_{i+1} \mid 1\leq i\leq n-1\}$ is an even base of $\Deltaa$
and $\Psi_\odd:=\Psi_\eve$ is a \adm odd base of $\Deltaa$.
Note that, $\dim(\q(n)_\odd^\alpha)=1$ for all $\alpha\in\Deltaa\setminus\{\mathbf{0}\}$.
\item\label{ex:bases(3)}
For the periplectic supergroup $\P(n)$,
we define $\Upsilon(\lambda_i):=n-i+1$ for each $1\leq i\leq n$.
Then $\Psi_\eve=\{\lambda_i-\lambda_{i+1} \mid 1\leq i\leq n-1\}$ is an even base of $\Deltaa$
and $\Psi_\odd=\{2\lambda_n\}$ is a \adm odd base of $\Deltaa$.
\item\label{ex:bases(4)}
For a Chevalley supergroup $\G$ of classical type, the existence of a Chevalley basis as in \cite{FioGav12} does not by itself imply the existence of a \adm base in the sense of Definition~\ref{asu:gen}.
The three conditions in that definition must be checked for the chosen positive roots.
\item\label{ex:bases(5)}
We consider the algebraic supergroup $\F^{\langle g_1,\dots,g_n\rangle}=F\ltimes (\Gao)^n$ given in Example~\ref{ex:root-sys}\eqref{ex:root-sys(4)}.
Suppose that $F=\GLL_n$ with standard split maximal torus $T$ and $\Upsilon(\lambda_i):=-i$ for each $1\leq i\leq n$.
Then $\Deltaa_\eve^\pm=\{\pm(\lambda_i-\lambda_j)\mid 1\leq i<j\leq n\}$.
For each $i$, write $\chi_i=m_i(\lambda_1+\cdots+\lambda_n)$ with $m_i\in\mathz$.
We have
\[
\Deltaa_\odd^+=\{-\chi_i\mid m_i>0\},
\qquad
\Deltaa_\odd^-=\{-\chi_i\mid m_i<0\}.
\]
The indices with $m_i=0$ contribute to $\h_\odd$.
Take $\Psi_\eve=\{\lambda_i-\lambda_{i+1}\mid 1\leq i<n\}$ and $\Psi_\odd=\Deltaa_\odd^+$.
Then Definition~\ref{asu:gen}\eqref{asu:gen(2)} holds since all positive odd root spaces are included among the generators.
For $\gamma\in\Psi_\odd$ and $\alpha\in\Psi_\eve$, the weight $\gamma-\alpha$ is neither an even root nor a determinant weight, so it does not belong to $\Deltaa$.
Also, $\Psi_\eve\cap\Psi_\odd=\varnothing$.
Thus $(\Psi_\eve,\Psi_\odd)$ is a \adm base of $\Deltaa$.
\qedd
\end{enumerate}
\end{ex}

\begin{rem}\label{rem:bases}
We give some remarks on \adm odd bases of root systems.
\begin{enumerate}
\item
We explain how a \adm odd base depends on the choice of $\Upsilon$.
Suppose that $\G=\P(2)$ and $\Upsilon(\lambda_i)=-i$ for $i=1,2$.
Then $\Deltaa^+=\{\lambda_1-\lambda_2,-(\lambda_1+\lambda_2)\}$ and $\Deltaa^-=\{-(\lambda_1-\lambda_2),\lambda_1+\lambda_2,2\lambda_1,2\lambda_2\}$.
Take $\Psi_\eve=\{\lambda_1-\lambda_2\}$ and $\Psi_\odd=\{-(\lambda_1+\lambda_2)\}$.
The only positive odd root space is included among the generators in Definition~\ref{asu:gen}\eqref{asu:gen(2)}.
Moreover, $-(\lambda_1+\lambda_2)-(\lambda_1-\lambda_2)=-2\lambda_1\notin\Deltaa$, and $\Psi_\eve\cap\Psi_\odd=\varnothing$.
Thus this pair is a \adm base.
Its odd part differs from the choice $\{2\lambda_2\}$ in Example~\ref{ex:bases}\eqref{ex:bases(3)}, which uses a different $\Upsilon$.
\item
If $\Psi_\odd$ is a \adm odd base of $\Deltaa$,
then by Definition~\ref{asu:gen}\eqref{asu:gen(2)},
we get
\[
\Deltaa^+=\Deltaa_\eve^+ \cup \Deltaa_\odd^+
\;\subset\; \mathz_{\geq0}\Psi_\eve + \mathz_{\geq0}\Psi_\odd,
\]
where $\mathz_{\geq0}:=\mathbb{N}\cup\{0\}$ and
$\mathz_{\geq0}\Psi_\epsilon:=\{\sum_i c_i \alpha_i \mid c_i\in\mathz_{\geq0}, \alpha_i\in\Psi_\epsilon\}$ ($\epsilon\in\mathz_2$).
The converse does not hold in general.
\qedd
\end{enumerate}
\end{rem}

\begin{lemm} \label{lem:even-act}
Suppose that $\Deltaa$ has a \adm base $(\Psi_\eve,\Psi_\odd)$.
Let $\lambda\in\X(T)^\flat$ and $\alpha\in\Psi_\eve\setminus\Psi_\odd$.
If $n\geq \langle\lambda,\alpha^\vee\rangle +1$,
then $X_{-\alpha}^{(n)} \acthy v^\lambda=0$ for all $v^\lambda\in L(\lambda)^\lambda$.
\end{lemm}
\begin{proof}
We use a weight bound which holds for every $\alpha\in\Psi_\eve$.
Set $\ell:=\langle\lambda,\alpha^\vee\rangle$.
The set of $T$-weights of the $\Gev$-module $L(\lambda)$ is stable under the Weyl group of $(\Gev,T)$.
If $\lambda-n\alpha$ were a weight for some integer $n>\ell$, its image under the reflection $s_\alpha$ would also be a weight.
However,
\[
s_\alpha(\lambda-n\alpha)=\lambda+(n-\ell)\alpha,
\]
which is excluded by the highest weight property of $\lambda$.
Thus $L(\lambda)^{\lambda-n\alpha}=0$ for $n>\ell$.
Since $X_{-\alpha}^{(n)}\acthy v^\lambda$ belongs to this weight space, the claim follows.
\end{proof}

If $\Psi_\eve\cap\Psi_\odd\neq\varnothing$,
then we put $K_\alpha:=[X_\alpha,Y_{-\alpha}]$ ($\in \h_\odd$) for each $\alpha\in\Psi_\eve\cap\Psi_\odd$.
For the notation $Y_{-\alpha}$, see the end of Section~\ref{sec:root-sys}.
\begin{deff} \label{def:rest-weight}
Suppose that $\Deltaa$ has a \adm base $(\Psi_\eve,\Psi_\odd)$.
An element $\lambda\in \X(T)^\flat$ is called the \defnote{$p^r$-restricted weight for $\G$} if it satisfies
the following conditions for all $\alpha\in\Psi_\eve$:
\begin{enumerate}
\item\label{def:rest-weight(1)}
For the case when $\alpha\notin\Psi_\odd$.
Then $\langle \lambda, \alpha^\vee \rangle \leq p^r-1$.
\item\label{def:rest-weight(2)}
For the case when $\alpha\in\Psi_\odd$.
If $\lambda([K_\alpha,K_\alpha])\neq0$ in $\Bbbk$, then $\langle\lambda,\alpha^\vee\rangle\leq p^r$.
If $\lambda([K_\alpha,K_\alpha])=0$, then $\langle\lambda,\alpha^\vee\rangle\leq p^r-1$.
\end{enumerate}
The set of all $p^r$-restricted weights for $\G$ are denoted by $\X_r(T)^\flat$.
\end{deff}

\begin{rem} \label{rem:BK}
Suppose that $\Bbbk$ is algebraically closed.
If $\G=\GL(m|n)$ and $\Q(n)$,
then the above $\X_{r=1}(T)^\flat$ coincides with $X_p^+(T)$ and $\X_p^+(T)_{\mathsf{res}}$ defined in \cite{Kuj03} and \cite{BruKle03}, respectively.
\qedd
\end{rem}

Let $V$ be a left $\G_r$-supermodule.
Recall that the induced action of $u\in \hy(\G_r)$ on $v\in V$ is denoted by $u\acthy v$, see \eqref{eq:acthy}.
For a super-subspace $W\subset V$, we write $\hy(\G_r)\acthy W$ for the $\Bbbk$-linear span of $\{u\acthy w\mid u\in\hy(\G_r),\ w\in W\}$.
The next is a key-lemma in this paper
whose proof is analogous to the proof of \cite[Lemma~9.8]{BruKle03}:
\begin{lemm} \label{prp:key-lemma}
Suppose that $\Deltaa$ has a \adm base $(\Psi_\eve,\Psi_\odd)$.
For each $\lambda\in\X_r(T)^\flat$, $\hy(\G_r) \acthy L(\lambda)^\lambda$ forms a left $\G$-supermodule.
In particular, $L(\lambda) = \hy(\G_r) \acthy L(\lambda)^\lambda$.
\end{lemm}
\begin{proof}
Set $q:=p^r$ and $M:=\hy(\G_r)\acthy L(\lambda)^\lambda$.
Since $\hy(\G_r)$ is stable under conjugation by $T$, the super-subspace $M$ is $T$-stable.
By Theorem~\ref{thm:G=loc-fin-hy(G)-T}, it is enough to prove that $M$ is $\hy(\G)$-stable.
The superalgebra $\hy(\G)$ is generated by $\hy(\Gev)$ and $\hy(\G_r)$, so we only need to prove stability under $\hy(\Gev)$.

For $x\in\hy(\Gev)$ and $u\in\hy(\G_r)$, the adjoint action satisfies
\[
xu=\sum_x (x_{(1)}\triact u)x_{(2)}.
\]
Since $\G_r$ is normal in $\G$, Proposition~\ref{prp:hy(normal)} gives $x_{(1)}\triact u\in\hy(\G_r)$.
The highest weight super-subspace is stable under $\hy(T)$ and is annihilated by the augmentation ideal of $\hy(\Uevp)$.
Using the coproducts of the divided powers and the generators of $\hy(\Gev)$ given at the beginning of this subsection, it is therefore enough to show that
\begin{equation}\label{eq:enough}
X_{-\alpha}^{(n)}\acthy v^\lambda\in M
\quad\text{for all}\quad
\alpha\in\Psi_\eve,\ n\geq p^r,\ v^\lambda\in L(\lambda)^\lambda.
\end{equation}
Indeed, the terms with exponent less than $p^r$ already belong to $\hy(\G_r)$.
If every divided power sends $L(\lambda)^\lambda$ into $M$, the preceding identity shows that each such generator preserves $M$.

Fix $\alpha\in\Psi_\eve$ and put $\ell:=\langle\lambda,\alpha^\vee\rangle$.
The proof of Lemma~\ref{lem:even-act} gives $L(\lambda)^{\lambda-n\alpha}=0$ whenever $n>\ell$, also when $\alpha\in\Psi_\odd$.
If $\alpha\notin\Psi_\odd$, then $\ell\leq q-1$ by Definition~\ref{def:rest-weight}\eqref{def:rest-weight(1)}, so \eqref{eq:enough} follows.
Suppose that $\alpha\in\Psi_\eve\cap\Psi_\odd$.
Set $a:=\lambda([K_\alpha,K_\alpha])$ and let $c:=a^{-1}$ if $a\neq0$, and $c:=1$ if $a=0$.
Put $K:=2cK_\alpha$.
Then $K_\alpha K=c[K_\alpha,K_\alpha]$.
We prove the following equality for every $m\in\mathbb{N}\cup\{0\}$.
\begin{equation}\label{eq:w}
X_{-\alpha}^{(p^r+m)}\acthy v^\lambda
=X_{-\alpha}^{(p^r+m-1)}Y_{-\alpha}K\acthy v^\lambda.
\end{equation}
Both sides have weight $\lambda-(q+m)\alpha$.
They are therefore zero when $q+m>\ell$.
By Definition~\ref{def:rest-weight}\eqref{def:rest-weight(2)}, the only remaining case is $m=0$, $\ell=q$ and $a\neq0$.

In this case, set $z:=Y_{-\alpha}K\acthy v^\lambda$.
The vector $K\acthy v^\lambda$ belongs to $L(\lambda)^\lambda$.
For $j\geq1$, the coproduct formula gives
\[
X_\alpha^{(j)}\acthy z
=(X_\alpha^{(j)}\triact Y_{-\alpha})K\acthy v^\lambda.
\]
For $j=1$, the right hand side is $K_\alpha K\acthy v^\lambda=ca\,v^\lambda=v^\lambda$.
For $j\geq2$, the element $X_\alpha^{(j)}\triact Y_{-\alpha}$ belongs to the odd weight space $\g_\odd^{(j-1)\alpha}$, since the adjoint action of $\Gev$ preserves $\g_\odd$.
This is either zero or a positive odd root space, so it annihilates $K\acthy v^\lambda$.
Thus $X_\alpha\acthy z=v^\lambda$ and $X_\alpha^{(j)}\acthy z=0$ for $j\geq2$.
Applying \eqref{eq:comm-formula}, we obtain
\[
\begin{aligned}
X_\alpha^{(q)}X_{-\alpha}^{(q)}\acthy v^\lambda
&=\binom{q}{q}v^\lambda=v^\lambda,\\
X_\alpha^{(q)}X_{-\alpha}^{(q-1)}\acthy z
&=\binom{q-1}{q-1}X_\alpha\acthy z=v^\lambda.
\end{aligned}
\]
Put $w:=X_{-\alpha}^{(q)}\acthy v^\lambda-X_{-\alpha}^{(q-1)}\acthy z$.
Then $X_\alpha^{(q)}\acthy w=0$.
For every $j>0$, the vector $X_{-\alpha}^{(j)}\acthy w$ has weight $\lambda-(q+j)\alpha$ and is zero by the same weight bound.
Using \eqref{eq:comm-formula} with $\alpha$ replaced by $-\alpha$, we get
\[
0=X_{-\alpha}^{(q)}X_\alpha^{(q)}\acthy w
=\binom{-H_\alpha}{q}\acthy w
=\binom{q}{q}w=w.
\]
Here the toral binomial coefficient is evaluated on the integral coroot weight $\langle\lambda-q\alpha,\alpha^\vee\rangle=-q$.
This proves \eqref{eq:w}.

For $m=0$, the right hand side of \eqref{eq:w} belongs to $M$, since $X_{-\alpha}^{(q-1)}$, $Y_{-\alpha}$ and $K$ belong to $\hy(\G_r)$.
For $m>0$, both sides are zero by the weight bound.
Thus \eqref{eq:enough} holds, and the argument above proves that $M$ is a left $\G$-supermodule.
It is non-zero, so the simplicity of $L(\lambda)$ gives $M=L(\lambda)$.
\end{proof}

\subsection{Steinberg's tensor product theorem} \label{sec:simple-G-G_r-supermods}
Throughout the rest of the paper,
we assume that the root system $\Deltaa$ of $\G$ (with respect to $T$) has a \adm base $(\Psi_\eve,\Psi_\odd)$, see Definition~\ref{asu:gen}.

As we have seen in Section~\ref{sec:simple-G-supermods},
{\it not} all simple supermodules are absolutely simple, in general
(see Proposition~\ref{prp:T=Tev-abs-simple} and Example~\ref{ex:non-abs-irr}).
Thus, we also assume the following condition on our base field $\Bbbk$.
\begin{asum} \label{asu:alg-cld}
If $\mathbf{0}\in\Deltaa$ (or equivalently, $\T\neq T$),
then we assume that the base field $\Bbbk$ is algebraically closed.
\end{asum}

We naturally regard a left $\G$-supermodule as a left $\G_r$-supermodule via the inclusion $\G_r\subset \G$ as before.
The following proof is due to Brundan and Kleshchev \cite[Lemma~9.6]{BruKle03}:
\begin{lemm} \label{prp:alg-cld}
Any simple left $\G$-supermodule is semisimple as a left $\G_r$-supermodule.
\end{lemm}
\begin{proof}
If $\mathbf{0}\notin\Deltaa$,
then to prove the claim, we may assume that $\Bbbk$ is algebraically closed by Propositions~\ref{prp:T=Tev-abs-simple} and \ref{prp:T=Tev-abs-simple_r}.
Otherwise, by Assumption~\ref{asu:alg-cld}, $\Bbbk$ is supposed to be algebraically closed.

Let $L$ be a simple left $\G$-supermodule.
Since $L\neq0$, we get $\soc_{\G_r}(L)\neq0$ (see \cite[Lemma~A.3]{Shi20}).
Thus, we fix a simple left $\G_r$-super-submodule $S$ of $L$.
For each $g\in\Gev(\Bbbk)$, the translate $g.S:=\{g.v\in L\mid v\in S\}$ is a simple left $\G_r$-super-submodule of $L$, since $\G_r$ is a normal super-subgroup of $\G$.
Thus, $M:=\sum_{g\in\Gev(\Bbbk)}g.S$ is a semisimple left $\G_r$-super-submodule of $L$.
In particular, $M$ is stable under $\hy(\G_r)$.

On the other hand, it is obvious that $M$ is a left $\Gev(\Bbbk)$-submodule of $L$.
Since $\Bbbk$ is algebraically closed and $\Gev$ is reduced,
$M$ is actually a left $\Gev$-submodule of $L$, see \cite[Part~I, Section~2.8]{Jan03}.
Thus, again by Theorem~\ref{thm:G=loc-fin-hy(G)-T}, $M$ is also a locally finite left $\hy(\Gev)$-$T$-submodule of $L$.
By Theorem~\ref{prp:PBW-Gr}, as a superalgebra, $\hy(\G)$ is generated by $\hy(\Gev)$ and $\hy(\G_r)$.
This shows that $M$ is actually a locally finite left $\hy(\G)$-$T$-supermodule,
and hence a left $\G$-supermodule again by Theorem~\ref{thm:G=loc-fin-hy(G)-T}.
Since $L$ is simple, we get $L=M$.
The proof is done.
\end{proof}

\begin{prop} \label{prp:key-prp}
For $\lambda\in\X_r(T)^\flat$,
we have $\hy(\G_r) \acthy L(\lambda)^\lambda \cong L_r(\lambda)$ as left $\G_r$-supermodules.
In particular, $L(\lambda) \cong L_r(\lambda)$ as left $\G_r$-supermodules.
\end{prop}
\begin{proof}
Since $L(\lambda)^\lambda \cong \uu(\lambda)$ as left $\B_r^+$-supermodules,
we get
\[
0\neq {}_{\B_r^+}\Hom\big( \uu(\lambda), L(\lambda) \big) \cong {}_{\G_r}\Hom\big( M_r(\lambda), L(\lambda)\big).
\]
Thus, the following is a non-zero surjective homomorphism of $\G_r$-supermodules:
\[
\varphi : M_r(\lambda) \longrightarrow \hy(\G_r) \acthy L(\lambda)^\lambda;
\quad u\otimes_{\hy(\B_r^+)} v \longmapsto u\acthy v,
\]
where $u\in\hy(\G_r)$, $v\in \uu(\lambda)$.
Since the quotient $M_r(\lambda)/\Ker(\varphi)$ is semisimple $\G_r$-supermodule by Lemmas~\ref{prp:key-lemma} and \ref{prp:alg-cld},
the radical of $M_r(\lambda)$ is contained in the kernel of $\varphi$.
This shows that there exists a surjective homomorphism
\[
\topp_{\G_r}(M_r(\lambda)) \longrightarrow M_r(\lambda)/\Ker(\varphi) \cong \hy(\G_r)\acthy L(\lambda)^\lambda.
\]
of left $\G_r$-supermodules.
This map is actually bijective,
since $L_r(\lambda)\cong \topp_{\G_r}(M_r(\lambda))$ by Proposition~\ref{prp:L_r=M_r/Rad}.
Also, by Lemma~\ref{prp:key-lemma}, we get $\hy(\G_r) \acthy L(\lambda)^\lambda = L(\lambda)$.
The proof is done.
\end{proof}

\begin{rem} \label{rem:need}
Under Assumption~\ref{asu:alg-cld}, Proposition~\ref{prp:T=Tev-abs-simple_r} and Schur's lemma imply that the even endomorphism ring of a simple left $\G_r$-supermodule is $\Bbbk$.
This property is used in the evaluation-map argument in the proof of Theorem~\ref{thm:main}.
\qedd
\end{rem}

For $\lambda\in\X(T)$, set $H_\ev^0(\lambda) := \ind_{\Bev}^{\Gev}(\Bbbk^\lambda)$.
Then it is known that
\[
\X(T)_+
:= \{ \lambda\in \X(T) \mid H_\ev^0(\lambda) \neq 0 \}
= \{ \lambda\in \X(T) \mid \forall \alpha\in\Deltaa_\eve^+,\,\, \langle\lambda,\alpha^\vee\rangle\geq 0 \}
\]
and the following map is bijective, see \cite[Part~II, Chapter~2]{Jan03}:
\begin{equation} \label{eq:L_ev}
\X(T)_+ \longrightarrow \Simple(\Gev);
\quad
\lambda \longmapsto L_\ev(\lambda):=\soc_{\Gev}(H_\ev^0(\lambda)).
\end{equation}
In \cite[Proposition~4.18]{Shi20}, it is shown that $\X(T)^\flat \subset \X(T)_+$.
Thus, for each $\lambda\in \X(T)^\flat$, we may consider $L_\ev(\lambda)$.

\begin{thm}\label{thm:main}
Let $\lambda\in\X(T)^\flat$.
Suppose that there exists $\lambda'\in\X_r(T)^\flat$ and $\mu\in\X(T)^\flat$ such that $\lambda=\lambda'+p^r\mu$.
Then there exists an isomorphism $L(\lambda)\cong L(\lambda')\otimes L_\ev(\mu)^{[r]}$ of left $\G$-supermodules.
\end{thm}
\begin{proof}
The highest weight superspace gives a non-zero map $M_r(\lambda)\to L(\lambda)$.
Its image is semisimple by Lemma~\ref{prp:alg-cld}, so Proposition~\ref{prp:L_r=M_r/Rad} identifies this image with $L_r(\lambda)$.

By Propositions~\ref{prp:L_r(lambda+p^rmu)} and \ref{prp:key-prp},
we have $L(\lambda')\cong L_r(\lambda')$ and $L(\lambda)\supset L_r(\lambda)\cong L_r(\lambda')$
as left $\G_r$-supermodules.
Thus,
\[
H := {}_{\G_r}\Hom(L(\lambda'), L(\lambda)) = {}_{\G_r}\underline{\Hom}(L(\lambda'), L(\lambda))_\eve
\]
is non-zero.
Then by Example~\ref{ex:[r]},
the following ``evaluation map'' is a morphism in the category of left $\G$-supermodules:
\[
\varphi : H\otimes L(\lambda') \longrightarrow L(\lambda);
\quad
f\otimes v \longmapsto f(v).
\]
Since $H\neq0$ and $L(\lambda)$ is a simple left $\G$-supermodule,
this $\varphi$ is surjective.
By Lemma~\ref{prp:alg-cld}, $L(\lambda)$ is semisimple as a left $\G_r$-supermodule.
Set $S:=L_r(\lambda')$.
Counting simple supermodules up to even isomorphism, we may write
$L(\lambda)|_{\G_r}\cong S^{\oplus m}\oplus N$,
where $m\geq0$, $N$ is semisimple and ${}_{\G_r}\Hom(S,N)=0$.

If $\mathbf{0}\notin\Deltaa$,
then as in the proof of Proposition~\ref{prp:T=Tev-abs-simple} (for $\G_r$),
one sees that ${}_{\G_r}\End(L_r(\lambda'))=\Bbbk$.
Otherwise, by Schur's lemma (see Assumption~\ref{asu:alg-cld}),
we also obtain the same result ${}_{\G_r}\End(L_r(\lambda'))=\Bbbk$.
Therefore, we conclude that $m$ coincides with $\dim(H)$ in both cases.
This shows that $\dim(H\otimes L(\lambda'))=\dim(H)\dim(L_r(\lambda')) \leq \dim (L(\lambda))$,
and hence $\varphi$ is actually an isomorphism.

To complete the proof, we shall show $H=L_\ev(\mu)^{[r]}$.
Since $\G_r$ acts trivially on $H$, we have $H=(H^{[-r]})^{[r]}$.
Thus, it is enough to see that $H^{[-r]}=L_\ev(\mu)$.
By the isomorphism $H\otimes L(\lambda')\cong L(\lambda)$ of left $\G$-supermodules,
we see that $H$ is a simple left $\G$-supermodule of ``highest'' weight $p^r\mu$.
In particular, $H^{[-r]}$ is a simple left $\Gev$-module of ``highest'' weight $\mu$,
and hence $H^{[-r]}$ must be isomorphic to $L_\ev(\mu)$ by \eqref{eq:L_ev}.
The proof is done.
\end{proof}

Combining Theorem~\ref{thm:main} with the ordinary Steinberg tensor product theorem, we obtain the following result:
\begin{cor}\label{cor:main}
Let $\lambda\in\X(T)^\flat$.
Suppose that
$\lambda=\lambda_0+p\lambda_1+\cdots+p^m\lambda_m$
for some $\lambda_0,\lambda_1,\dots,\lambda_m\in\X_1(T)^\flat$, and that
\[
0\leq\langle\lambda_i,\alpha^\vee\rangle<p
\quad\text{for all }1\leq i\leq m\text{ and }\alpha\in\Psi_\eve.
\]
Then there exists an isomorphism
\[
L(\lambda) \cong L(\lambda_0)\otimes L_\ev(\lambda_1)^{[1]} \otimes \cdots \otimes L_\ev(\lambda_m)^{[m]}
\]
of left $\G$-supermodules.
\end{cor}
\begin{proof}
The case $m=0$ is immediate, so suppose that $m\geq1$ and put
$\mu:=\lambda_1+p\lambda_2+\cdots+p^{m-1}\lambda_m$.
We first note that $\X(T)^\flat$ is closed under addition.
Indeed, for $\xi,\eta\in\X(T)^\flat$, the tensor product $L(\xi)\otimes L(\eta)$ has the weight $\xi+\eta$, and all its weights are less than or equal to $\xi+\eta$ in the positive-root order.
A simple composition factor containing this weight therefore has highest weight $\xi+\eta$.
Thus, $\xi+\eta\in\X(T)^\flat$.
It follows that $\mu\in\X(T)^\flat$, so Theorem~\ref{thm:main} gives
\[
L(\lambda)\cong L(\lambda_0)\otimes L_\ev(\mu)^{[1]}.
\]
By the additional restriction on $\lambda_1,\dots,\lambda_m$, the ordinary Steinberg tensor product theorem \cite[Part~II, Corollary~3.17]{Jan03} yields
\[
L_\ev(\mu)\cong L_\ev(\lambda_1)\otimes L_\ev(\lambda_2)^{[1]}
\otimes\cdots\otimes L_\ev(\lambda_m)^{[m-1]}.
\]
Taking its first Frobenius twist proves the claim.
\end{proof}

\providecommand{\bysame}{\leavevmode\hbox to3em{\hrulefill}\thinspace}
\providecommand{\MR}{\relax\ifhmode\unskip\space\fi MR }
\providecommand{\MRhref}[2]{%
  \href{http://www.ams.org/mathscinet-getitem?mr=#1}{#2}
}
\providecommand{\href}[2]{#2}

\end{document}